\documentclass[10pt,leqno]{amsart}

\usepackage{amssymb}
\usepackage{amsmath}
\usepackage{amsxtra}
\usepackage{amscd}
\usepackage{amsfonts}
\usepackage[utf8]{inputenc}
\usepackage{hyperref}
\usepackage{color}
\definecolor{lightgrey}{rgb}{0.9,0.9,0.9}
\usepackage{geometry}

\usepackage{graphicx}

\newtheorem{theorem}{Theorem}
\newtheorem{corollary}{Corollary}
\newtheorem{lemma}{Lemma}
\newtheorem{proposition}{Proposition}
\newtheorem{definition}{Definition}
\newtheorem{remark}{Remark}

\newtheorem*{note}{Note}

\newcommand{\Y}{{\rm Y}_{d,n}^{\mathtt B}}
\newcommand{\Z}{{\rm Y}_{d,n-1}^{\mathtt B}}

\newcommand{\F}{{\rm FTL}_{d,n}^{\mathtt B}}

\newcommand{\s}{\mathtt s}
\newcommand{\R}{\mathtt r}
\newcommand{\T}{{\rm TL}_n^{\mathtt{B}}}
\newcommand{\U}{\mathsf{u}}
\newcommand{\V}{\mathsf{v}}

\numberwithin{equation}{section}

\allowdisplaybreaks[1]

\makeatletter
\let\@wraptoccontribs\wraptoccontribs
\makeatother

\begin{document}

\title{Framization of a Temperley-Lieb algebra of type $\mathtt{B}$}

\author{M. Flores}
\address{Instituto de Matem\'{a}ticas \\ Universidad de Valpara\'{i}so \\ Gran Breta\~{n}a 1091, Valpara\'{i}so, Chile}
\email{marcelo.flores@uv.cl}

\author{D. Goundaroulis}
\address{Center for Integrative Genomics,
University of Lausanne,
1015 Lausanne, Switzerland.}
\address{Swiss Institute of Bioinformatics, 1015, Lausanne, Switzerland.}
\email{dimoklis.gkountaroulis@unil.ch}



\begin{abstract}
We extend the Framization of the Temperley-Lieb algebra to Coxeter systems of
type $\mathtt{B}$. We first define a natural extension of the classical
Temperley-Lieb algebra to Coxeter systems of type $\mathtt{B}$ and prove that such an
extension supports a unique linear Markov trace function. We then introduce the Framization of the Temperley-Lieb algebra of type
$\mathtt{B}$ as a quotient of the Yokonuma-Hecke algebra of type $\mathtt{B}$. The main theorem provides necessary and sufficient
conditions for the Markov trace  defined on the Yokonuma-Hecke algebra of type $\mathtt{B}$ to pass to the quotient algebra. Using
the main theorem, we construct invariants for framed links and classical links inside the solid torus.
\end{abstract}

\maketitle

\section{Introduction}

The Temperley-Lieb algebra appeared  originally in the study of the Potts model in statistical mechanics and in the
ice-type model in two dimensions \cite{tl71}. In the 1980's the Temperley-Lieb algebra was rediscovered by Jones in the context of
von Neumann
algebras \cite{jo83} and later as a quotient of the Hecke algebra \cite{jo}. The Hecke algebra supports a unique inductive linear
trace that can be
rescaled according to the Markov equivalence for braids and under certain conditions it passes to the Temperley-Lieb algebra. This
procedure leads to the definition of the Jones  polynomial. For these reasons, the
Hecke algebra and the Temperley-Lieb algebra are often considered as knot algebras. Another notable example of a knot algebra is the
BMW algebra \cite{bw, mu}.

Framization is  a technique introduced by Juyumaya and
Lambropoulou that produces new knot algebras associated to framed knots and links  \cite{jula}. Framization adds new generators,
called {\it the framing generators}, to the generating set of a known knot algebra and defines relations between the original and the
framing generators of the algebra. From an algebraic point of view, a knot algebra might have multiple candidates that are valid.
However, since the motivation of the technique is to obtain new polynomial invariants for (framed) links, candidates that produce
new, non-trivial link invariants are preferred. In particular, when multiple framization candidates for a knot algebra are
considered, the framization of the algebra that is most natural from a topological point of view is chosen \cite{gojukola2}.

A basic example of framization is the Yokonuma-Hecke algebra of type $\mathtt{A}$, denoted ${\rm Y}_{d,n}(u)$. It was introduced in
the context of Chevalley groups in
\cite{yo} and can be regarded as the framization of the Hecke algebra. Juyumaya fine-tuned  the presentation of ${\rm Y}_{d,n}(u)$ by
giving a natural description in terms of the framed braid group \cite{jusur}.  In recent years, framizations of several knot algebras have appeared \cite{jula2,jula,jula4,jula5,gojukola} that led to Jones-type invariants for framed \cite{jula}, classical \cite{jula, chjukala}, and singular links \cite{jula3}.

The Framization of the Temperley-Lieb algebra  ${\rm FTL}_{d,n}(q)$ was introduced in
\cite{gojukola2} as a quotient of ${\rm Y}_{d,n}(u)$. From this, a family of one-variable invariants for classical links in $S^3$,
denoted $\theta_d
(q)$, was derived  by finding the necessary and sufficient conditions for the trace of ${\rm Y}_{d,n}(u)$
to pass to the quotient algebra. For $d=1$, the invariant $\theta_1$ coincides with the Jones polynomial while for $d\neq1$,
$\theta_d$ is {\it not } topologically equivalent to the Jones polynomial {\it on links} \cite{gojukola2}.
More recently, Goundaroulis and Lambropoulou generalized the invariants $\theta_d (q)$ to a new two-variable invariant that is
stronger than the Jones polynomial on links and that can also detect the Thistlethwaite
link \cite{gola}.

All the results that are mentioned above are related to the Coxeter group of type $\mathtt{A}$. However, there is a growing
interest in framizations of algebras that are related to Coxeter systems of type $\mathtt{B}$. Indeed, the affine and
cyclotomic Yokonuma-Hecke algebras were introduced in \cite{chpoIMRN}, while in \cite{fjl} Flores and collaborators introduced $\Y (\U, \V)$,
the Yokonuma-Hecke algebra of type $\mathtt B$.

In this paper we extend the  Framization of the Temperley-Lieb algebra of type $\mathtt{A}$ to Coxeter
groups of type $\mathtt{B}$ by implementing the methods of \cite{fjl}. We first consider the generalized Temperley-Lieb
algebra that is associated to an arbitrary Coxeter system \cite{grlo}, and specialize it to the case of Coxeter systems of type
$\mathtt{B}$. We denote this algebra by ${\rm TL}_n^{\mathtt{B}}(\U, \V)$ and show that it emerges naturally as a quotient of the
Hecke algebra of type $\mathtt{B}$, denoted ${\rm H}_n(\U, \V)$. We then compute the necessary and sufficient conditions for the
Markov trace of ${\rm H}_n(\U, \V)$
\cite{la} to pass to the quotient algebra. The Framization of the Temperley-Lieb algebra of type $\mathtt{B}$, which is denoted $\F
(\U, \V)$, is  defined  as a quotient of the algebra $\Y (\U ,\V)$. For $d=1$, the algebra $\F (\U, \V)$ coincides with ${\rm
TL}_n^{\mathtt{B}}(\U, \V)$. The
main theorem determines the necessary and sufficient conditions such that the trace of $\Y (\U, \V)$ \cite{fjl} passes to $\F (\U,
\V)$. Finally, we investigate the conditions of the main theorem which generate topologically non-trivial invariants for framed
and classical links and we define those invariants.

The outline of the paper is as follows. In Section~\ref{prelim} we introduce the notation and we present the
classical braid group, the framed braid group, the algebra ${\rm H}_n(\U, \V)$, its framization
$\Y(\U, \V)$, and the Framization of the Temperley-Lieb algebra of type $\mathtt{A}$. In Section~\ref{sectlb} we introduce the
Temperley-Lieb algebra associated to the Coxeter group of type ${\mathtt B}$, denoted ${\rm TL}_n^{\mathtt{B}} (\U, \V)$. We also
determine the necessary and sufficient conditions  such that the trace on ${\rm H}_n ( \U , \V)$ passes to the algebra ${\rm TL}_n^{\mathtt{B}} (\U, \V)$ and
construct the corresponding link invariants. In Section~\ref{secftlb} we present the algebra $\F (\U ,\V)$ as a quotient of the
algebra $\Y (\U, \V)$ modulo an appropriate two-sided ideal and determine
the necessary and sufficient conditions so that the Markov trace defined on the algebra $\Y (\U, \V)$ passes  to the
quotient algebra. In Section~\ref{secinv} we use the trace on $\F (\U, \V)$ to define invariants
for framed and classical links and provide a set of skein relations for both cases. Finally, we show that the invariants for
classical links from $\F (\U, \V)$ are stronger than the Jones polynomial in the solid torus since they distinguish more pairs of
affine links.

\section{Preliminaries}\label{prelim}
Let $\U$, $\V$ be indeterminates. With the term algebra we mean an associative algebra with unity over $\mathbb{K}:=\mathbb{C}(\U,
\V)$.
\subsection{{\it Groups of type} ${\mathtt B}_n$}
For $n\geq 2$, we define the Coxeter group of type ${\mathtt  B}_n$, denoted by $W_n$, as the finite Coxeter group associated to the
following Dynkin diagram:
 \begin{center}
\setlength\unitlength{0.2ex}
\begin{picture}(350,40)
\put(82,20){$\R_1$}
\put(120,20){$\s_{1}$}
\put(200,20){$\s_{n-2}$}
\put(240,20){$\s_{n-1}$}

\put(85,10){\circle{5}}
\put(87.5,11){\line(1,0){35}}
\put(87.5,9){\line(1,0){35}}
\put(125,10){\circle{5}}
\put(127.5,10){\line(1,0){10}}

\put(145,10){\circle*{2}}
\put(165,10){\circle*{2}}
\put(185,10){\circle*{2}}

\put(205,10){\circle{5}}
\put(207.5,10){\line(1,0){35}}
\put(245,10){\circle{5}}
\put(192.5,10){\line(1,0){10}}


\end{picture}
\end{center}
Let $\R_k=\s_{k-1}\ldots \s_1 \R_1 \s_1\ldots \s_{k-1}$ for $2\leq k\leq n$. Every element
$w\in W_n$ can be written uniquely in a reduced expression as follows \cite{gela}: $w=w_1\ldots w_n$ with $w_k\in \mathtt{N}_k$,
$1\leq k\leq
n$, where
\begin{equation}\label{NWn}
\mathtt{N}_k:=\left\{
1, \R_{k},
\s_{k-1}\cdots \s_{i},
\s_{k-1}\cdots \s_{i}\R_{i}\, ;\, 1\leq i \leq k-1
\right\}.
\end{equation}
The \textit{braid group of type} ${\mathtt  B}_n$ associated to $W_n$, is defined as the group $\widetilde{W}_n$
generated  by $\rho_1 , \sigma_1 ,\ldots ,\sigma_{n-1}$ subject to the following relations
 \begin{equation}\label{braidB}
\begin{array}{rcll}
 \sigma_i \sigma_j & =  & \sigma_j \sigma_i & \text{ for} \quad \vert i-j\vert >1 ,\\
  \sigma_i \sigma_j \sigma_i & = & \sigma_j \sigma_i \sigma_j & \text{ for} \quad \vert i-j\vert = 1 ,\\
   \rho_1\sigma_i&=&\sigma_i\rho_1 &\text{ for}\quad i>1 ,\\
\rho_1 \sigma_1 \rho_1\sigma_1 & = & \sigma_1 \rho_1 \sigma_1\rho_1. &
  \end{array}
\end{equation}

Geometrically, braids of type ${\mathtt  B}_{n}$ can be viewed as classical braids of type ${\mathtt  A}_{n+1}$ with $n+1$ strands,
where the first strand is identically fixed and is called `the fixed strand'. The 2nd, \ldots, $(n+1)$st strands are renamed
from 1 to $n$ and they are called `the moving strands'. The `loop' generator $\rho_1$ corresponds to the looping of the first moving
strand around the fixed strand in the right-handed sense (see Fig.~\ref{genalg}).\smallbreak
The {\it $d$-modular framed braid group of type} $\mathtt{B}_n$ is defined as follows:
 \[
\mathcal{F}^{\mathtt{B}}_{d, n}: =\left (  C_d \right)^n \rtimes \widetilde{W}_n,
 \]
where $C_d:=\langle t\ |\ t^d=1\rangle$, is the cyclic group of order $d$, and the action of $\widetilde{W}_n$ on  $  C_d$ is given
by: $ t_j \sigma_i = \sigma_i t_{\s_i(j)}$ and $t_i \rho_1 = \rho_1 t_i$, for  $1 \leq i \leq n$. In both cases $t_i$ is the element
of $ (C_d)^n$ that has $t$ in the $i^{th}$ position and $1$ everywhere else. For $1\leq i,j \leq n$ and $m\in \{0,\ldots,d-1\}$, we
define the following elements on $\mathbb{C} \mathcal{F}^{\mathtt{B}}_{d,n}$:
  \begin{equation}\label{eifi}
 e_{i,j}^{(m)} = \frac{1}{d} \sum_{s=0}^{d-1} t_i^{m+s}t_j^{-s} \quad \mbox{and} \quad f_i= \frac{1}{d} \sum_{k=0}^{d-1}
 t_i^{k}.
 \end{equation}
For  $j=i+1$, we denote  $e_i^{(m)}:=e_{i,i+1}^{(m)}$ and $e_{i,j} : = e_{i,j}^{(0)}$. Note that $e_{i,j}^{(m)}$ and $f_i$ are
idempotent elements.

\subsection{\it{The Hecke algebra of type ${\mathtt B}$}\label{HeckeB}}
The Hecke algebra of type $\mathtt{B}$, denoted by ${\rm H}_n(\U, \V)$, can be considered as the quotient of
$\mathbb{K}[{\widetilde W}_n]$ modulo the two-sided ideal that is generated by the following elements:
\[
\sigma_i^2 - (\U-\U^{-1})\sigma_i -1  \quad \mbox{and} \quad \rho_1^2 - (\V - \V^{-1})\rho_1 -1.
\]
In terms of generators and relations, ${\rm H}_n(\U, \V)$ is the algebra that is generated by the elements $b_1,
g_1, \ldots, g_{n-1}$ which are subject  to the following relations:
\[
\begin{array}{cccl}
g_ig_j & = &  g_j g_i & \text{for all}\quad \vert i - j\vert >1,\\
g_ig_{i+1}  g_i  & = &  g_{i+1} g_i  g_{i+1}&\text{for all}\quad i =1, \ldots , n-2 ,\\
g_1  b_1 g_1 b_1 & = & b_1 g_1 b_1 g_1 , & \\
g_i^2 & = & 1 + (\U-\U^{-1})g_i &  \text{for all}\quad i ,\\
b_1^2 & = & 1 + (\V-\V^{-1})b_1.&
\end{array}
\]
The dimension of ${\rm H}_n(\U, \V)$ is $2^nn!$ and  for $\U = \V = 1$ it  coincides with $\mathbb{K}[W_n]$.
Consider now the following subsets of ${\rm H}_n(\U, \V)$:
\[
\mathtt{M}_1 = \{ 1, \, b_1 \}, \quad \mathtt{M}_2  =\{ 1, \,b_2 ,\, g_1, \,g_1 b_1 \}, \ \ldots ,  \
\mathtt{M}_n  =  \{ 1,\, b_n , \,\, g_{n-1}x \ | \ x  \in \mathtt{M}_{n-1} \}.
\]
where $b_k:=g_{k-1}\dots g_1 b_1 g_1^{-1}\dots g_{k-1}^{-1}$, for all $2\leq k\leq n$. The following set is a linear basis for
the algebra ${\rm H}_n(\U, \V)$:
\begin{equation}\label{Hnbas}
{\mathtt C}_n = \left \{ m_1 m_2 \ldots m_n  \ \vert \ m_i \in \mathtt{M}_i \right\}.
\end{equation}
There exists a natural epimorphism $\widetilde{W}_n \rightarrow {\rm H}_n(\U, \V)$ sending $\sigma_i
 \mapsto g_i$ and $\rho_1 \mapsto  b_1$. Additionally, the Hecke algebra of type ${\mathtt B}$  supports a unique Markov trace
 function \cite{gela}. Indeed, for any indeterminate $z,y$ there exists a linear trace:
\[
{\rm \tau} : \cup_{n=1}^{\infty} {\rm H}_n(\U,\V) \rightarrow \mathbb{K}[z,y]
\]
that is defined inductively by the following four rules:
\[
\begin{array}{lllll}
(1)& {\rm \tau} (\mathbf{1}_{n+1})&= 1, &\mbox{for all } n&\\
(2)& {\rm \tau} (a b) &= {\rm \tau} (ba), & a,b \in {\rm H}_n(q) &(\mbox{Conjugation property})\\
(3)&{\rm \tau} (a g_n ) &= z \, {\rm \tau}(a), & a\in {\rm H}_n(q) &(\mbox{Markov property for braiding generators})\\
(4)&{\rm \tau} (a b_{n+1} ) &= y \, {\rm \tau}(a), & a\in {\rm H}_n(q) &(\mbox{Markov property for looping generator}),
\end{array}
\]

\begin{remark}\label{rem1} \rm
A different presentation is  often used for the algebra ${\rm H}_n(\U,\V) $ that involves parameters $q$ and
$Q$, as well as different quadratic relations. More precisely, the quadratic relations are the following:
\[
(\overline{g}_i)^2 = (q-1)\overline{g}_i +q \quad \mbox{and} \quad (\overline{b}_1)^2 = (Q-1)\overline{b}_1 + Q.
\]
One can switch between the two presentations by taking $ \overline{g}_i =\U g_i$, $
\overline{b}_1= \V b_1$, $q = \U^2$ and $Q= \V^2$.
\end{remark}
By introducing the term $ \lambda = \frac{z - ( \U - \U^{-1})}{z}$,
one can re-scale $\tau$ so that it satisfies the braid equivalence in the solid torus \cite[Theorem~3]{la}. By
normalizing $\tau$, link invariants in the solid torus can be defined. Indeed, we have \cite[Definition~1]{la}:
\begin{equation}\label{homflyb}
P^{\mathtt{B}}(\U, \V, z, y)(\widehat\alpha) = \left ( \frac{1 - \lambda}{\sqrt \lambda ( \U - \U^{-1})} \right)^{n-1} \left (\sqrt
\lambda \right )^{\varepsilon(\alpha)} \tau\left( \pi ( \alpha ) \right),
\end{equation}
where $\widehat\alpha$ is the closure of the braid $\alpha$ inside the solid torus, $\pi$ is the natural epimorphism
$\widetilde{W}_n \rightarrow {\rm H}_n(\U, \V)$, and $\varepsilon(\alpha)$ is the algebraic sum of the exponents of
the braiding generators in $\alpha$. Furthermore, the invariant $P^{\mathtt{B}}$ can be defined completely by the following two
skein relations:
\begin{align}
\frac{1}{\sqrt \lambda} \ P^{\mathtt{B}} (L_+) - \sqrt \lambda \ P^{\mathtt{B}} (L_{-})  &= \left ( \U -\U^{-1} \right)
P^{\mathtt{B}} (L_0) \label{skein1P} \\
P^{\mathtt{B}} (M_+)  - P^{\mathtt{B}} (M_-)& = \left ( \V - \V^{-1} \right ) P^{\mathtt{B}} (M_0)\label{skein2P} ,
\end{align}
where $L_+$, $L_{-}$, $L_0$, $M_+$ , $M_{-}$ and $M_0$ are as shown in Fig.~\ref{Mis}.
\begin{figure}[h]
  \centering
  \includegraphics{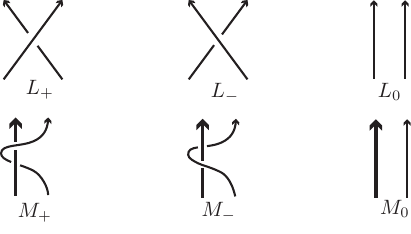}
  \caption{The elements $L_+$, $L_{-}$, $L_0$ constitute a Conway triple. The elements $M_+$ , $M_{-}$ and $M_0$ involve the fixed
  strand (shown in bold).}\label{Mis}
\end{figure}

\subsection{\it{The framization of the Hecke algebra of type ${\mathtt B}$.} }\label{framedB}  The framization of the Hecke
algebra of type $\mathtt{B}$ \cite{fjl}, denoted by ${\rm Y}_{d, n}^{\mathtt{B}} := {\rm Y}_{d, n}^{\mathtt{B}}(\U, \V)$, is defined
as the
algebra over $\mathbb{K}$ generated by the framing generators $t_1,\dots,t_n$, the braiding generators
$g_1,\dots,g_{n-1}$ and the loop generator $b_1$, subject to the following relations:
\begin{eqnarray}g_ig_j & = & g_jg_i  \quad \text{ for} \quad \vert i-j\vert > 1,\label{braid1}\\
  g_i g_j g_i & = & g_j g_i g_j \quad \text{ for} \quad \vert i-j\vert = 1, \label{braid2}\\
  b_1 g_i & = & g_i b_1 \quad \text{for all}\quad  i\not= 1, \label{braid3}\\
b_1 g_1 b_1 g_1 & = &  g_1 b_1 g_1 b_1, \label{braid4}\\
t_i t_j & =  & t_j t_i  \quad \text{for all }  \  i, j,\label{modular2}\\
  t_j g_i & =  & g_i t_{s_i(j)}\quad \text{for all }\,  i, j,  \label{th}\\
t_i b_1 & = & b_1 t_i \quad \text{for all i},\quad   \label{tb1}\\
t_i^d & = & 1 \quad \text{for all}\quad i, \label{modular1}\\
g_i^2 & = &  1+ (\U-\U^{-1})e_ig_i\quad \text{for all $i$}, \label{quadraticU}\\
b_1^2 & = &  1 + (\V-\V^{-1})f_1b_1.\label{quadraticV}
\end{eqnarray}
where $e_i$ and $f_1$ are as in \eqref{eifi}.  In Figure~\ref{genalg} we
illustrate the generators of the algebra $\Y$.

\begin{figure}[h]
  \centering
  \includegraphics{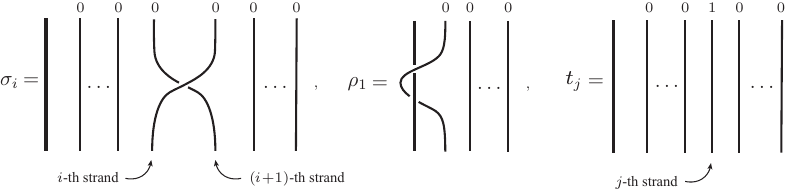}
  \caption{The generators of $\Y (\U, \V)$.}\label{genalg}
\end{figure}

\begin{note}\rm
For $d=1$, the algebra ${\rm Y}_{1,n}^{\mathtt B}$ coincides with ${\rm H}_n(\U,\V)$. By mapping $g_i\mapsto g_i$ and
  $t_i\mapsto 1$, we obtain an epimorphism from $\Y$ to ${\rm H}_n(\U,\V)$. Moreover, if we map the $t_i$'s  to a fixed non-trivial
  $d$-th root of the unity, we have an epimorphism from $\Y$ to ${\rm H}_n(\U,1)$.
\end{note}

In \cite{fjl} two different linear bases for $\Y$ are given, denoted by $\mathsf{D}_n$ and $\mathsf{C}_n$ respectively. We only
recall the second one, since it is the one that is used in the definition of  the Markov trace of $\Y$. For all
$1\leq k\leq n$, we define inductively  the sets $M_{d,k}$ by:
$$
M_{d,1} = \{t_1^m, t_1^m b_1 \,;\, 0 \leq m \leq d-1\}
$$
and
$$
M_{d,k}=\{t_k^m, t_k^mb_{k}, g_{k-1} x\,;\, x \in M_{d,k-1},\, 0\leq m\leq d-1\}
\quad \quad \mbox{for all $2\leq k\leq n$.}
$$
where the elements $b_k$'s are as in Section~\ref{HeckeB}. Define now $\mathsf{C}_n$ as the subset of $\Y$ formed by the elements
  $\mathfrak{m}_1\mathfrak{m}_2\cdots \mathfrak{m}_n$, with $\mathfrak{m}_i\in M_{d,i}$. Moreover, every element of $M_{d,k}$ has the
  form $\mathfrak{m}_{k,j,m}^+$ or
$\mathfrak{m}_{k,j,m}^-$ with $j\leq k$ and $0\leq m\leq d-1$,  where
$$
\mathfrak{m}_{k,k,m}^+:=t_k^m, \qquad \mathfrak{m}_{k,j,m}^+ := g_{k-1}\cdots g_jt_j^m\quad \text{for}\ j<k,
$$
and
$$
\mathfrak{m}_{k,k,m}^-:=t_k^mb_k, \qquad \mathfrak{m}_{k,j,m}^- := g_{k-1}\cdots g_jb_{j}t_j^m\quad \text{for}\ j<k.
$$
From the above, one can deduce that the basis $\mathsf{C}_n$ for $\Y$  may be rewritten as follows \cite[Proposition 5]{fjl}):
\begin{equation}\label{}
  \mathsf{C}_n=\{ t_1^{a_1}t_2^{a_2}\dots t_n^{a_n}m \ |\  m \in \mathtt{C}_n,\ a_i\in \{0,\ldots, d-1\}\}.
\end{equation}

In \cite{fjl} Flores et al. proved that $\Y$ supports a unique Markov trace. In brief,
they construct a certain family of linear maps ${\rm tr}_n: \Y \longrightarrow {\rm
Y}_{d,n-1}^{\mathtt{B}}$, called {\it relative traces}, that build step by step the desired Markov properties  (see also
\cite{chpoIMRN}). Finally, the  Markov trace on $\Y$ is defined by:  $${\rm Tr}_n := {\rm tr}_1\circ \cdots \circ {\rm tr}_n.$$
\begin{theorem}[cf. Theorem~3 \cite{fjl}]\label{trace}
Let $z,x_1,\dots, x_{d-1}, y_0, \dots ,y_{d-1}$ be indeterminates in  $\mathbb{K}(z, x_1,\dots, x_{d-1},$ $y_0, \dots ,y_{d-1})$ and
let
$x_0:=1$. Then the linear map  ${\rm Tr}$ is a Markov trace on $\{\Y \}_{n\geq 1}$. That is, for every $n\geq 1$,  the linear map
${\rm Tr}_n: \Y \longrightarrow \mathbb{K}(z, x_1,\dots, x_{d-1}, y_0, \dots ,y_{d-1})$ satisfies  the following rules:
   \begin{enumerate}
    \item[(i)] ${\rm Tr}_n(1)=1$,
    \item[(ii)] ${\rm Tr}_{n+1}(Xg_n)=z{\rm Tr}_n(X)$,
    \item[(iii)] ${\rm Tr}_{n+1}(Xb_{n+1}t_{n+1}^{m})=y_{m}{\rm Tr}_n(X)$,
    \item[(iv)] ${\rm Tr}_{n+1}(Xt_{n+1}^{m})=x_{m}{\rm Tr}_n(X)$,
    \item[(v)] ${\rm Tr}_n(XY)={\rm Tr}_n(YX),$
 \end{enumerate}
  where $X,Y\in \Y$.
\end{theorem}

Recall that the method of Jones for obtaining link invariants requires a rescaled
and normalized Markov trace function.  An interesting property of the trace ${\rm Tr}$ is that it does not rescale directly according
to the framed braid equivalence for
the solid torus. Indeed, the trace ${\rm Tr}$ can be rescaled only if the parameters $x_i$, $1 \leq i \leq
d-1$, are solutions of a non-linear system of equations that is called the ${\rm E}$-system \cite[Appendix]{jula}, while the
parameters $y_j$, $0 \leq j \leq d-1$, are solutions of an analogous non-linear system called the ${\rm F}$-system \cite{fjl}.
Consequently, new invariants for framed knots and links in the solid torus can be constructed,
denoted by  ${\mathcal X}_S^{\mathtt{B}}$, that are parametrized by $S\subseteq  C_d$ (for  more details see
\cite[Section 7]{fjl}). The invariants ${\mathcal X}_S^{\mathtt{B}}$ when restricted to framed links with all framings equal to zero,
give rise to invariants of oriented classical links in the solid torus. Since classical knot
theory embeds in the knot theory of the solid torus and by using the results of \cite{chjukala}, we deduce that  the ${\mathcal
X}_S^{\mathtt{B}}$  invariants are different than the invariant $P^{\mathtt{B}}(\U , \V
,x,y)$ on {\it links} \cite{gela, la}.

\subsection{{\it The Framization of the Temperley-Lieb algebra of type $\mathtt{A}$}}\label{TLframedtypeA}
The framization of the Temper\-ley-Lieb algebra of type $\mathtt{A}$ and the derived invariants
for framed and classical links were studied extensively by Goundaroulis and collaborators \cite{go, gojukola, gojukola2}.
As mentioned earlier, it is a well known fact that the Temperley-Lieb algebra of type $\mathtt{A}$ can be obtained as the quotient
of the algebra ${\rm H}_n(u)$ modulo the two-sided ideal that is generated by the following elements:
 \[
 g_{i,i+1} := \sum_{w \in \langle s_i, s_{i+1} \rangle} g_w.
\]
Similarly, the framization of the Temperley-Lieb algebra of type $\mathtt{A}$ is defined as a quotient of the
Yokonuma-Hecke algebra of type $\mathtt{A}$, which is denoted by ${\rm Y}_{d,n}(u)$ \cite{ju}. However, such a
quotient is not unique in the case of framization. As mentioned in the introduction, the quotient algebra that eventually is chosen
is the most natural with respect
to the construction of new, non-trivial invariants for framed and classical knot and links.

The first quotient algebra that was studied is the Yokonuma-Temperley-Lieb algebra \cite{gojukola}, denoted ${\rm YTL}_{d,n}(u)$,
and proved to
be too restrictive. As a consequence, basic pairs of framed links were not distinguished. For this reason this algebra
was discarded as a potential candidate for the framization of the Temperley-Lieb algebra however, the Jones polynomial was
recovered from this construction. The second candidate was the {\it Complex Reflection Temperley-Lieb algebra}, denoted ${\rm
CTL}_{d,n}(u)$ \cite{gojukola2}. In contrast to the case of ${\rm YTL}_{d,n}(u)$, the invariants that are derived from ${\rm
CTL}_{d,n}(u)$ proved to coincide either with those from the algebra ${\rm Y}_{d,n}(u)$ or with those that are derived from the
actual framization of the Temperley-Lieb algebra \cite[Proposition~10]{gojukola2}. This result is consistent with the fact that the
algebra ${\rm CTL}_{d,n}(u)$ is isomorphic to a direct sum of matrix algebras over tensor products of Temperley-Lieb and
Iwahori-Hecke algebras \cite{ChPou2}. Thus, the quotient algebra ${\rm CTL}_{d,n}(u)$ is also discarded as a potential candidate
for the framization of the Temperley-Lieb algebra.

The framization of the Temperley-Lieb algebra is an intermediate algebra between the algebras ${\rm YTL}_{d,n}(u)$ and ${\rm
CTL}_{d,n}(u)$. It is denoted by ${\rm FTL}_{d,n}(u)$,
and it is defined as the quotient of the algebra ${\rm Y}_{d,n}(u)$ modulo the two-sided ideal that is generated by the element:
\[
r_{1,2} :=  e_1 e_2 (1+g_1 + g_{2} + g_{1} g_{2} + g_{2} g_1 + g_1 g_{2}
g_1).
\]

In \cite[Theorem~6]{gojukola2} necessary and sufficient conditions were determined so that  the trace of ${\rm Y}_{d,n}(u)$ passes to
${\rm
 FTL}_{d,n}(u)$. These conditions led to a family of new 1-variable invariants for classical links, $\left\{
 \theta_d \right\}_{d\in \mathbb{N}}$, that are topologically not equivalent to the Jones polynomial on links, while they are
 topologically equivalent to the Jones polynomial on knots \cite[Theorem~9]{gojukola2}.  Finally, the invariants $\theta_d(q)$ can be
 generalized to a 2-variable invariant for classical links,
 $\theta(q,E)$. More precisely, we have the following:
 \begin{theorem}[{\cite[Theorem~1.1]{gola}}]
Let $q, E$ be indeterminates and let  $\mathcal{L}$ be the set of all oriented links. There exists a unique ambient isotopy
invariant of classical oriented links
\[
\theta : \mathcal{L} \rightarrow \mathbb{C}[q^{\pm 1} , E^{\pm 1}]
\]
defined by the following rules:
\begin{enumerate}
\item On crossings involving different components the following skein relation holds:
\[
q^{-2}\, \theta (L_+) - q^2\, \theta (L_-) = (q - q^{-1})\, \theta (L_0),
\]
where $L_+$, $L_-$ and $L_0$ constitute a Conway triple.
\item For a  union $\mathcal{K} = \sqcup_{i=1}^r K_i$ of $r$ unlinked knots, with $r\geq1$, it holds that:
\[
\theta (\mathcal{K}) = E^{1-r} V(\mathcal{K}),
\]
where $V(\mathcal{K})$ is the value of the Jones polynomial on $\mathcal{K}$.
\end{enumerate}

\end{theorem}
The invariant $\theta(q,E)$ is topologically equivalent to the Jones polynomial on knots while it is stronger than the Jones
polynomial on links \cite[Theorem~5]{gola} (see Section~\ref{sec:geom}).

\section{The Temperley-Lieb algebra associated to the Coxeter group of type ${\mathtt B}$}\label{sectlb}
 We begin this section by defining the Temperley-Lieb algebra  of type ${\mathtt B}$  as a
 quotient of the Hecke algebra of type $\mathtt{B}$. This is derived from the definition for  an arbitrary Coxeter group
 \cite{grlo}.\smallbreak

As mentioned earlier, the classical Temperley-Lieb algebra can be expressed as a quotient of the Hecke algebra of type $\mathtt{A}$.
Based on this, Fan and Green defined the Temperley-Lieb algebras associated to any simply laced Coxeter group \cite{fangr}. This was
done by first considering the Hecke algebra associated to the respective Coxeter group and then naturally extending the defining
ideal of the classical case. Using the same procedure Green and Losonczy extended this definition to any Coxeter group \cite{grlo}.
Specifically, consider $(W,S)$ to be
an arbitrary Coxeter System, and let $H(W)$ be the associated Hecke algebra.  Then, the algebra $H(W)$ has a basis  consisting of
elements $T_w$, $w \in W$ that satisfy:
\begin{equation}\label{coxmult}
  T_{s} T_{w}= \left\{\begin{array}{lcl}
                    T_{sw}, & \mbox{if} & \ell(sw) > \ell (w)   \\
                   a_s T_{sw} + b_s T_w,& \mbox{if} & \ell(sw) < \ell(w)
                  \end{array}\right.
\end{equation}
where  $\ell$ is the length function in $W$ and $a_s$, $b_s$ are parameters that depend on $s\in S$ such that $a_s = a_t$ and $b_s
= b_t$ whenever $s$ and $t$ are conjugate in $W$. For further details the reader is referred to \cite[Chapter 7]{hump}. Let
now $J$ be the two-sided ideal of $H(W)$ that is generated by the following elements:
$$\sum_{w\in \langle s_i, s_j \rangle}T_w$$
where $(s_i, s_j)$ runs over all pairs of $S$ that correspond to adjacent nodes in the Dynkin diagram of $W$. Then the generalized
Temperley-Lieb algebra, ${\rm TL}(W)$, is defined as the quotient $H(W)/J$.

We shall specialize now the algebra ${\rm TL}(W)$ to the case of Coxeter systems of type $\mathtt{B}$. From the discussion above
and by considering also the change of generators in Remark~\ref{rem1}, we have that the defining two-sided ideal, denoted by
$J_B$, is generated by the elements:
   \begin{eqnarray*}
    g_{i,i+1} &=& 1 + \U(g_i + g_{i+1})+ \U^2(g_ig_{i+1} + g_{i+1}g_i) + \U^3 g_ig_{i+1}g_i\\
    g_{\mathtt{B}} &:=& 1 + \U g_1 + \V b_1 + \U\V (g_1b_1+b_1g_1)+ \U^2\V g_1b_1g_1 +
    \V^2\U b_1g_1b_1 \\
    &&+ (\U\V)^2 g_1b_1g_1b_1,
  \end{eqnarray*}
where $1\leq i \leq n-2$. Given that the elements $g_{i,i+1}$  are all conjugates of $g_{1,2}$ in ${\rm H}_n(\U, \V)$ (see
  \cite{gojukola}), we conclude that $J_B=\langle g_{\mathtt{B}}, g_{1,2}\rangle$.

\begin{definition}
  We define $\T:={\rm TL}_n^{\mathtt{B}}(\U,\V)$, the Temperley-Lieb algebra associated to the Coxeter group of type ${\mathtt B}$
  as the quotient ${\rm H}_n(\U, \V)/J_B$.
\end{definition}

\subsection{{\it A Markov trace on the algebra ${\rm TL}_n^\mathtt{B}$}}\label{traceHB}

The purpose of this section is to find the necessary and  sufficient conditions such that the trace defined in ${\rm H}_n(\U,\V)$
passes to $\T$. \smallbreak

 Let $W$ be a Coxeter group, and $H(W)$ the Hecke algebra associated to $W$. Now consider $b_s=a_s-1$ in (\ref{coxmult})  and set
 $x=\sum_{w\in W} T_w$. Observe that \cite[Lemma 3.2]{mathas}  is valid for every
finite Coxeter group, that is:
\begin{equation}\label{Coxrel}
  xT_s=a_sx, \quad \hbox{for all $s\in S$}
\end{equation}
Equation (\ref{Coxrel}) and direct computations prove the following two lemmas.

\begin{lemma}\label{multA}
The following holds in $H_n(\U,\V)$:
\begin{itemize}
  \item[i)] $g_1 g_{1,2}= g_{1,2}g_1=\U g_{1,2}$
  \item[ii)] $g_2 g_{1,2}= g_{1,2}g_2=\U g_{1,2}$
\end{itemize}
\end{lemma}

\begin{lemma}\label{multB}
In ${\rm H}_n(\U, \V)$ the following equations holds
\begin{itemize}
  \item[i)] $b_1 g_{\mathtt{B}}= g_{\mathtt{B}}b_1=\V g_{\mathtt{B}}$
  \item[ii)] $g_1 g_{\mathtt{B}}= g_{\mathtt{B}}g_1=\U g_{\mathtt{B}}$
\end{itemize}
  \end{lemma}

In analogy to ${\rm TL}_n(u)$, the trace ${\rm \tau}$ passes to the
quotient $\T$ if and only if ${\rm \tau}$ annihilates the defining ideal $\langle g_{1,2} , g_{\mathtt{B}} \rangle $ of $\T$:
\begin{equation}\label{iffeq}
 {\rm \tau} (m g_{\mathtt{B}} ) + {\rm \tau} (n g_{1,2} )= 0,
\end{equation}
where $m,n$ are in the linear basis of ${\rm H}_n (\U, \V)$. We shall determine now the necessary and sufficient
conditions so that \eqref{iffeq} holds. We will use induction on $n$. We start with the
following lemma:

\begin{lemma}\label{trh12}
The following hold in ${\rm H}_n (\U, \V)$:
\begin{align*}
{\rm \tau}(g_{1,2})& = (\U^2+1) (\U z)^2 + (\U^2+2)\U z +1 \\
{\rm \tau}(g_{\mathtt{B}})&=  \U^2\V^2 y^2 + ( \U \V + \U^3 \V^3) zy + (\V + \U^2 \V)y + ( \U + \U^3 \V^2)z +1
\end{align*}
\end{lemma}

\begin{proof}
The proof follows immediately from the defining rules of ${\rm \tau}$.
\end{proof}

We shall treat each summand of \eqref{iffeq} separately. For the first summand we have the following:
\begin{proposition}\label{relationtypeb}
  For all $m\in {\rm H}_n(\U,\V)$ we have that

 $${\rm \tau} (mg_{\mathtt{B}})= p(\U,\V,z,y){\rm \tau}(g_{\mathtt{B}}),\quad \text{for all $n\geq 2$}$$
where $p(\U,\V,z,y)$ is a monomial in the variables $\U,\V,z, y$.
\end{proposition}

\begin{proof}
  By linearity of the trace, it is enough to prove the statement for an element $m$ in the inductive basis $\mathtt{C}_n$. We will
  proceed by induction. For $n=2$ the result follows directly by Lemma \ref{multB}. Suppose now that the argument holds for any
  $w\in {\rm H}_n(\U,\V)$, and let $m\in {\rm H}_{n+1}(\U,\V)$, where $m=wb_{n+1}$ or $m=wg_n\dots g_ib_i^a$, with $a=0,1$ and
  $w\in {\rm H}_n(\U,\V)$. We have that
  \begin{eqnarray*}
    {\rm \tau}(wb_{n+1}) &=& y{\rm \tau}(w) \\
    {\rm \tau}(wg_n\dots g_ib_i^a)  &=& z{\rm \tau}(\alpha)
  \end{eqnarray*}
where $\alpha=wg_{n-1}\dots g_ib_i^a \in {\rm H}_n(\U,\V)$ and so the result follows by the induction hypothesis.
\end{proof}

\begin{lemma}\label{h1,2b}
  For $i\geq1$ we have that
 $${\rm \tau}(b_ig_{1,2})=y{\rm \tau}(g_{1,2})  $$

\end{lemma}
\begin{proof}
  First note that ${\rm \tau}(b_1g_{1,2})=y{\rm \tau}(g_{1,2})$ follows easily by the trace rules, since $b_1 \in {\rm
  H}_1(\U,\V)$. For $i=2$ we have that
  $${\rm \tau}(b_2g_{1,2})={\rm \tau}(g_1b_1g_1^{-1}g_{1,2}).$$
From Lemma \ref{multA} we obtain
  $${\rm \tau}(g_1b_1g_1^{-1}g_{1,2})=\U^{-1} {\rm \tau}(g_1b_1g_{1,2})=\U^{-1} {\rm \tau}(b_1g_{1,2}g_1)= {\rm
  \tau}(b_1g_{1,2})=y{\rm \tau}(g_{1,2}).$$
The case $i=3$ is completely analogous, while for $i \geq 4$ the result follows immediately by the trace rules.
\end{proof}
The following proposition deals with the second term of \eqref{iffeq}.
\begin{proposition}\label{proptypeA}
  Let $n\geq 3$. For all $m\in {\rm H}_n(\U,\V)$ we have that
\begin{equation}\label{relationtypea}
  {\rm \tau}(mg_{1,2})= \left\{\begin{array}{l}
                    p(\U,\V,z,y){\rm \tau}(g_{1,2})   \\
                    p(\U,\V,z,y){\rm \tau}(b_1g_1b_1g_{1,2})   \\
                    p(\U,\V,z,y){\rm \tau}(b_1g_1b_1g_2g_1b_1g_{1,2})
                  \end{array}\right.
  \end{equation}
where $p(\U,\V,z,y)$ is a monomial in the variables $\U,\V,z, y$.
\end{proposition}
\begin{proof}
Since  $\tau$ is linear, it's enough to prove (\ref{relationtypea}) for any $m$ in the basis $\mathsf{C}_n$  from
${\rm H}_n(\U,\V)$. Again, we will use induction on $n$. We start by proving that the argument holds for $n=3$.
First note that
  $$\mathsf{C}_2=\{1,b_2,g_1,g_1b_1,b_1,b_1b_2,b_1g_1,b_1g_1b_1\}$$
From Lemmas~\ref{h1,2b} and \ref{multA} we have that
 $$ \begin{array}{rclrcl}
    {\rm \tau}(b_2g_{1,2}) & = &y {\rm \tau}(g_{1,2});\quad  &{\rm \tau}(g_1g_{1,2})  &  =&\U {\rm \tau}(g_{1,2})  \\
    {\rm \tau}(g_1b_1g_{1,2}) & = &\U y {\rm \tau}(g_{1,2})  ;\quad& {\rm \tau}(b_1g_{1,2}) & = & y {\rm \tau}(g_{1,2}) \\
    {\rm \tau}(b_1g_1g_{1,2}) & = & \U y {\rm \tau}(g_{1,2}) ;\quad&{\rm \tau}(b_1b_2g_{1,2})  & = &  \U^{-1}{\rm
    \tau}(b_1g_1b_1g_{1,2})
  \end{array}$$
Suppose now that $m\in C_3$. This means that $m=wm_3$ for some $w\in \mathsf{C}_2$ and $m_3\in \{g_2g_1b_1, g_2b_2, b_3, g_2g_1, g_2,
1\}$. From the previous results and Lemma~\ref{multA} we obtain that
$$\begin{array}{rclrcl}
  {\rm \tau}(wg_2b_2g_{1,2}) & = &\U^{-1}{\rm \tau}(wg_2g_1b_1g_{1,2}) ;\quad &{\rm \tau}(wb_3g_{1,2})  & = & \U^{-2}{\rm
  \tau}(wg_2g_1b_1g_{1,2}) \\
  {\rm \tau}(wg_2g_1g_{1,2}) & = &\U^2 {\rm \tau}(wg_{1,2}) ;\quad &{\rm \tau}(wg_2g_{1,2})  & = & \U {\rm \tau}(wg_{1,2})
\end{array}$$

Therefore, we only have to study ${\rm \tau}(wg_2g_1b_1g_{1,2})$. Replacing $w$ and by applying previous lemmas
and using the trace rules on each element in $\mathsf{C}_2$, we have:
$$ \begin{array}{rclrcl}
   {\rm \tau}(g_2g_1b_1g_{1,2})   & = & \U^2 y {\rm \tau}(g_{1,2}) ;\quad&  {\rm \tau}(b_2g_2g_1b_1g_{1,2})  & = &  \U {\rm
   \tau}(b_1g_1b_1g_{1,2}) \\
   {\rm \tau}(g_1g_2g_1b_1g_{1,2})     &  =&  \U^3 y {\rm \tau}(g_{1,2});\quad&  {\rm \tau}(g_1b_1 g_2g_1b_1g_{1,2})  & = & \U^2 {\rm
   \tau}(b_1g_1b_1g_{1,2})  \\
      {\rm \tau}(b_1g_2g_1b_1g_{1,2}) & = &  \U {\rm \tau}(b_1g_1b_1g_{1,2});\quad&  {\rm \tau}(b_1g_1g_2g_1b_1g_{1,2}) & = &  \U^2
      {\rm \tau}(b_1g_1b_1g_{1,2})
  \end{array}$$
$${\rm \tau}(b_1b_2g_2g_1b_1g_{1,2}) = \U^{-1}{\rm \tau}(b_1g_1b_1g_2g_1b_1g_{1,2})$$
From the above, the result follows for $n=3$. Finally, suppose that the argument holds for $m\in {\rm H}_n(\U,\V)$ and let $m\in
{\rm H}_{n+1}(\U,\V)$. We have that $m=wb_{n+1}$ or $m=wg_n\dots g_ib_i^a$, with $a=0,1$ and $w\in {\rm H}_n(\U,\V)$. Since we have
that
  \begin{eqnarray*}
    {\rm \tau}(wb_{n+1}) &=& y{\rm \tau}(w) \\
    {\rm \tau}(wg_n\dots g_ib_i^a)  &=& z{\rm \tau}(\alpha),\quad \text{where $\alpha=wg_{n-1}\dots g_ib_i^a \in {\rm H}_n(\U,\V)$}
  \end{eqnarray*}
  then result follows by the induction hypothesis.
\end{proof}
The discussion above suggests that \eqref{iffeq} reduces to a homogenous system of four equations of the trace parameters $z$ and
$y$, namely:
\begin{theorem}\label{equivstat}
  The following statements are equivalent
  \begin{itemize}
    \item[i)] ${\rm \tau}(mg_{1,2})+{\rm \tau}(ng_{\mathtt{B}})=0\quad \text{for all $m,n \in {\rm H}_n(\U,\V)$}$
    \item[ii)]${\rm \tau}(g_{\mathtt{B}})={\rm \tau}(g_{1,2})={\rm \tau}(b_1g_1b_1g_{1,2})={\rm
        \tau}(b_1g_1b_1g_2g_1b_1g_{1,2})=0  $
  \end{itemize}
  \end{theorem}
  \begin{proof}
  Since (i) holds for all $m,n \in {\rm H}_n(\U,\V)$, then it must also hold for $m=1$ and $n=0$. Thus, we deduce argument (ii). The
  converse is a  direct consequence of Propositions~\ref{relationtypeb} and \ref{proptypeA}.
\end{proof}

The following lemma will be used in the proof of Theorem~\ref{tracetlb} below. We have that:

\begin{lemma}\label{trb1h1b1}
The following equations hold:\[
\begin{array}{l l c l }
(i) & \tau (b_1 g_1 b_1 g_{1,2}) &=& \V^{-1} \left ( \U ( 1+ \U z + \U^3 z) ( \V y^2 + \U(\V + ( \V^2 - 1) y ) z ) \right )\\
(ii)& \tau( b_1 g_1 b_1 g_2 g_1 b_1 g_{1,2})& = &\V^{-2} \left ( \U^3 ( \V^2 y^3 + \U(2+ \U^2) \V y ( \V +( \V^2 -1) y )z
\right.\\
 & & &\left. + \U^2 ( 1+ \U^2) ( y + \V( \V^2 -1) (1+ \V y) )z^2 \right )
\end{array}
\]
\end{lemma}

\begin{proof}
The proof is a long straightforward computation using the rules of $\tau$.
\end{proof}

We are now able to give the necessary and sufficient conditions for $\tau$ to pass to $\T (\U,
\V)$. Indeed, we have:
\begin{theorem}\label{tracetlb}
The trace $\tau$ passes to the quotient algebra $\T (\U,\V)$ if and only if the trace parameters $z$ and $y$ take one of
the following values.

\begin{itemize}
\item[(i)] $z = -\frac{1}{\U}$ and $y = -\frac{1}{\V}$, \hspace{1.6cm} (ii)\ $z = -\frac{1}{\U}$ and $y = \V$,
\item[(iii)]\ $z = -\frac{1}{\U (1 + \U^2)}$ and  $y = -\frac{1}{\V}$,  \qquad (iv)\  $z = -\frac{1}{\U (1 + \U^2)}$ and $y =
    \frac{-1 + \V^2}{(1 + \U^2) \V}$.
\end{itemize}
\end{theorem}
\begin{proof}
From Theorem~\ref{equivstat} we have that $\tau$ annihilates the ideal $J$ if and only if the following system of equations has
solutions for $z$ and $y$:
\[
(\Sigma) : = \ \left \{ \begin{aligned}
&{\rm \tau}(g_{\mathtt{B}})=0\\
&{\rm \tau}(g_{1,2})=0\\
&{\rm \tau}(b_1g_1b_1g_{1,2})=0\\
&{\rm \tau}(b_1g_1b_1g_2g_1b_1g_{1,2})=0
\end{aligned}
\right.
\]
Using Lemmas~\ref{trh12} and \ref{trb1h1b1} one can derive the four sets of solutions for $(\Sigma)$ and, therefore, the necessary
and sufficient conditions for the passing of $\tau$ to $\T$.
\end{proof}

\subsection{\it Link invariants from $\T (\U, \V)$} Following Jones \cite{jo}, we can now define
link invariants in the solid torus. Starting from
\eqref{homflyb}, we specialize the parameters $z, y$ to the necessary and sufficient conditions of Theorem~\ref{tracetlb}. Note that
the values $z=-1/ \U$, $y=-1/ \V$ and
$y=-\V$ are discarded since they are of no topological interest \cite[Section~11]{jo}. From the remaining pair of values $z =
-\frac{1}{\U (1 + \U^2)}$ and $y = \frac{\V^2-1}{(1 + \U^2) \V}$ we deduce that $\lambda= \U^4$ and thus we have:
\begin{definition}
The following is an invariant for links inside the solid torus
\begin{equation}\label{vtypeb}
V^{\mathtt{B}}(\U,\V) := \left (-\frac{1+\U^2}{\U} \right)^{n-1} \U^{2\varepsilon(\alpha)} \tau ( \overline\pi (\alpha)),
\end{equation}
where $\alpha$, $\widehat \alpha$, $\varepsilon(\alpha)$ are as in \eqref{homflyb} and  $\overline\pi$ is  the natural epimorphism
$\widetilde{W}_n \rightarrow \T (\U, \V)$ sending $\sigma_i \mapsto g_i$ and $r_1 \mapsto  b_1$.
\end{definition}

\begin{remark} \rm By substituting $\lambda=\U^4$ in \eqref{skein1P} and \eqref{skein2P} we derive that  $V^{\mathtt{B}}$ can be
defined completely by the following skein relations:
\begin{align}
\U^{-2} \ V^{\mathtt{B}} (L_+) - \U^2 \ V^{\mathtt{B}} (L_{-})  &= \left ( \U -\U^{-1} \right) V^{\mathtt{B}} (L_0) \label{skein1V}
\\
V^{\mathtt{B}} ( M_+)  - V^{\mathtt{B}} ( M_-)& = \left ( \V - \V^{-1} \right ) V^{\mathtt{B}} ( M_0), \label{skein2V}
\end{align}
where $L_+$, $L_{-}$, $L_0$, $M_+$ , $M_{-}$ and $M_0$ are as shown in Fig.~\ref{Mis}.
\end{remark}

\section{Framization of the Temperley-Lieb algebra associated to the Coxeter group of type ${\mathtt B}$}\label{secftlb}
In this section we introduce ${\rm FTL}_{d,n}^{\mathtt{B}}$, the framization of the Temperley-Lieb algebra associated to the Coxeter
group of type ${\mathtt B}$. This extends naturally the work done for the type ${\mathtt A}$ case in \cite{gojukola2}. In more
detail, the framization will be defined as a quotient of the algebra $\Y$ modulo an appropriate two-sided ideal. Since the algebra
${\rm Y}_{d,n}(u)$ is contained in $\Y$, following Section~\ref{TLframedtypeA} we consider the following element in $\Y( \U, \V)$:
\begin{align}
&r_{1,2}:=e_1e_2 g_{1,2}.
\end{align}
The element $r_{1,2}$ is the generator of the type $\mathtt{A}$ part of the quotient algebra $\F$. Accordingly, we consider also
the generator of the type $\mathtt{B}$ part, which is the element:
\[
r_{{\mathtt B}}:=f_1e_1g_{\mathtt{B}}
\]

\begin{definition}\rm
The framization of the Temperley-Lieb algebra associated to the Coxeter group of type ${\mathtt B}$ is defined as follows:
\begin{align}
&\F (\U, \V):=\Y (\U , \V)/\langle r_{{\mathtt B}}, r_{1,2}\rangle
\end{align}
\end{definition}

We give below the framed analogues of Lemmas~\ref{multA} and \ref{multB} that will be used extensively later.
\begin{lemma}\label{multfraA}
The following holds in $\Y (\U, \V)$:
\begin{itemize}
  \item[i)] $g_1 r_{1,2}= r_{1,2}g_1=\U r_{1,2}$
  \item[ii)] $g_2 r_{1,2}= r_{1,2}g_2=\U r_{1,2}$
\end{itemize}
\end{lemma}

\begin{proof}
The proof follows from a straightforward computation. For demonstrative reasons, we  only prove the case $r_{1,2}g_1$. Observe that
the element $e_1e_2$ commutes with $g_1$ and $g_2$. On the other hand, we also have that
\begin{equation}\label{relA}
  e_ig_i^2=e_i(1+(\U-\U^{-1})e_ig_i)=e_i(1+(\U-\U^{-1})g_i).
\end{equation}

The result follows from (\ref{relA}) and by using Lemma~\ref{multA}.
\end{proof}
We will prove now that an analogous result holds for the generator of the $\mathtt{B}$-type case.
\begin{lemma}\label{multfraB}
In $\Y (\U, \V)$ the following equations holds
\begin{itemize}
  \item[i)] $b_1 r_{{\mathtt B}}= r_{{\mathtt B}}b_1=\V r_{{\mathtt B}}$
  \item[ii)] $g_1 r_{{\mathtt B}}= r_{{\mathtt B}}g_1=\U r_{{\mathtt B}}$
\end{itemize}
  \end{lemma}
\begin{proof}
 We only prove the left multiplication for the first case. The proof for the second case is analogous.  Similarly to the previous
 case, we have that the element $f_1f_2$ commutes with $b_1$ and $g_1$. Note now that the following equation holds in $\Y (\U, \V)$:
  $$  f_1b_1^2=f_1(1+(\V-\V^{-1})f_1b_1)=f_1(1+(\V-\V^{-1})b_1).$$
The result follows by using Lemma~\ref{multB}.
\end{proof}

\subsection{{\it Technical lemmas}}
Our next goal is to determine the necessary and sufficient conditions so that the trace ${\rm Tr}$ of $\Y (\U,
\V)$ in \cite{fjl} passes to  ${\rm FTL}^{\mathtt{B}}_{d,n}$. Our approach will be analogous to
\cite{gojukola2}. However, we need to postpone this discussion until the next subsection in order to present here a series of
technical results that  are required for the proof of our main theorem.
\begin{lemma}\label{r12gbtr}The following holds in ${\rm Y}_{d,n}^{\mathtt{B}}(\U,\V)$.
\begin{align}
1. &\ {\rm Tr}(r_{\mathtt{B}}) =
\frac{1}{d^2}\sum_{r,s}x_rx_s+\U^2\V^2\frac{1}{d^2}\sum_{r,s}y_ry_s+\V(\U^2+1)\frac{1}{d^2}\sum_{r,s}x_sy_r \label{trrb}.\\
 &\hspace{2.5cm}+ z\U[1+\U^2\V^2]\frac{1}{d}\sum_r x_r+ z[\U^3\V^3 + \U \V]\frac{1}{d}\sum_r y_r \nonumber\\
 2.& \ {\rm Tr}({r_{1,2}})= (u+1)z^2x_{m} + (u+2)z \, E^{(m)}  +{\rm tr}(e_1^{(m)}e_2) \label{trr12}.
\end{align}
\end{lemma}

\begin{proof}
For the first argument we have that:

\begin{align*}
   {\rm Tr}(r_{\mathtt{B}})=&\frac{1}{d^2}\sum_{r,s}x_rx_s+\V\frac{1}{d^2}\sum_{r,s}x_sy_r+z\U\frac{1}{d}\sum_r x_r +2z\U\V
   \frac{1}{d}\sum_r y_r +z\V^2\U\frac{1}{d}\sum_r x_r\\
   & +z\V^2\U(\V- \V^{-1})\frac{1}{d}\sum_r y_r+\U^2\V\frac{1}{d^2}\sum_{r,s}x_ry_s+ z\U^2\V(\U -\U^{-1})\frac{1}{d}\sum_r y_r\\
   &+\U^2\V^2\frac{1}{d^2}\sum_{r,s}y_ry_s+ z\U^2\V^2(\U- \U^{-1})\frac{1}{d}\sum_r x_r +z\U^2\V^2(\U- \U^{-1})(\V-
   \V^{-1})\frac{1}{d}\sum_r y_r\\
   =& \frac{1}{d^2}\sum_{r,s}x_rx_s+\U^2\V^2\frac{1}{d^2}\sum_{r,s}y_ry_s+(\U^2\V+\V)\frac{1}{d^2}\sum_{r,s}x_sy_r+ \\
   & z \left [\U+\V^2\U+\V^2\U^2(\U- \U^{-1}) \right ]\frac{1}{d}\sum_r x_r+z \left [2\U\V+\U^2\V(\U-\U^{-1})+ \U \V^2 ( \V -
   \V^{-1}) \right.\\
   &+ \left. \U^2\V^2(\U- \U^{-1})(\V- \V^{-1})\right ] \frac{1}{d}\sum_r y_r. \\
   =&\frac{1}{d^2}\sum_{r,s}x_rx_s+\U^2\V^2\frac{1}{d^2}\sum_{r,s}y_ry_s+\V(\U^2+1)\frac{1}{d^2}\sum_{r,s}x_sy_r+
   z\U[1+\U^2\V^2]\frac{1}{d}\sum_r x_r\\
   & +z[\U^3\V^3 + \U \V]\frac{1}{d}\sum_r y_r.
\end{align*}

For the second part of the proof the reader is referred to \cite[Lemma~7]{gojukola2}.
\end{proof}

The following two propositions show how ${\rm Tr}$ behaves on the elements of the defining ideal of  ${\rm
FTL}^{\mathtt{B}}(\U, \V)$. We start by exploring the case of the elements that involve the $\mathtt{B}$-type part of the algebra.

\begin{proposition}\label{propframetypeb}
  For all $m\in \Y$ we have that

 $${\rm Tr} (mr_{\mathtt{B}})= p(\U,\V,z,y_a,x_b){\rm Tr}(r_{\mathtt{B}}),\quad \text{for all $n\geq 2$,}$$
where $p(\U,\V,z,y_a,x_b)$ is a monomial in the variables $\U,\V,z, y_a$ and $x_b$, with $0\leq a,b\leq d-1$.
\end{proposition}

\begin{proof}
  By the linearity of the trace, it is enough to prove the statement for an element $m$ in the inductive basis $\mathsf{C}_n$. We will
  proceed by induction. For $n=2$ the result follows from Lemma \ref{multfraB}, and the fact that element $f_1f_2$ absorbs the
  framing part of $m$. For instance, if $m=t_1^at_2^bb_1g_1b_1$ we have
  $${\rm Tr} (mr_{\mathtt{B}})={\rm Tr} (t_1^at_2^bb_1g_1b_1f_1f_2g_{\mathtt{B}})=\V^2\U{\rm Tr}
  (t_1^at_2^bf_1f_2g_{\mathtt{B}})=\V^2\U {\rm Tr} (f_1f_2g_{\mathtt{B}}).$$

Suppose now that the argument holds for any $w\in \Z$, and let $m\in \Y$. Then, the element $m$ can be written as follows:
    $$m=wt_n^bb_{n}^a\quad \text{or} \quad  m=wg_{n-1}\dots g_ib_i^at_i^b,$$
  with $w\in \Z$, $a=0,1$ and $0\leq b\leq d-1$ . Since we have that
  \begin{eqnarray*}
    {\rm Tr}(wt_n^ab_{n}) &=& y_a{\rm Tr}(w) \\
    {\rm Tr}(wt_n^a) &=& x_a{\rm Tr}(w) \\
    {\rm Tr}(wg_{n-1}\dots g_ib_i^at_i^b)  &=& z{\rm Tr}(\alpha),\quad \text{where $\alpha=wg_{n-2}\dots g_ib_i^at_i^b \in {\rm
    H}_n(\U,\V)$}
  \end{eqnarray*}
  the result follows by the induction hypothesis.
\end{proof}

The next proposition deals with the $\mathtt{A}$-type part of the algebra.

\begin{proposition}\label{propframedtypeA}
  Let $n\geq 3$. For all $m\in \Y$ we have that
\begin{equation}\label{relationframedtypea}
  {\rm Tr}(mr_{1,2})= \left\{\begin{array}{l}
                    p(\U,\V,z,y_a, x_b){\rm Tr}(t_1^at_2^bt_3^c r_{1,2})\\
                    p(\U,\V,z,y_a, x_b){\rm Tr}(t_1^at_2^bt_3^cb_1r_{1,2})   \\
                    p(\U,\V,z,y_a, x_b){\rm Tr}(t_1^at_2^bt_3^cb_1g_1b_1r_{1,2})   \\
                    p(\U,\V,z,y_a, x_b){\rm Tr}(t_1^at_2^bt_3^cb_1g_1b_1g_2g_1b_1r_{1,2})
                  \end{array}\right.
  \end{equation}
where $p(\U,\V,z,y_a,x_b)$ is a monomial in the variables $\U,\V,z, y_a$ and $x_b$, with $0\leq a,b\leq d-1$.
\end{proposition}
\begin{proof}
The trace ${\rm Tr}$ is linear so it suffices to prove (\ref{relationframedtypea}) for any $m$ in the basis $\mathsf{C}_n$  of
$\Y$. We will use again induction on $n$. We start by proving that the argument holds
for $n=3$. Note that from Lemmas~\ref{multfraB} and \ref{multfraA} we have:

 \begin{align*}
    &{\rm Tr}(t_1^at_2^bt_3^cb_2r_{1,2}) =  {\rm Tr}(t_1^at_2^bt_3^cb_1r_{1,2})  ;\quad  {\rm Tr}(t_1^at_2^bt_3^cb_1b_2r_{1,2}) =
    \U^{-1}{\rm Tr}(t_1^at_2^bt_3^cb_1g_1b_1r_{1,2}) \\
    & {\rm Tr}(t_1^at_2^bt_3^cg_1b_1r_{1,2}) =  \U y {\rm Tr}(t_1^at_2^bt_3^cb_1r_{1,2}) ;\quad  {\rm
    Tr}(t_1^at_2^bt_3^cb_1g_1r_{1,2})  = \U y {\rm Tr}(t_1^at_2^bt_3^cb_1r_{1,2})  \\
    & {\rm Tr}(t_1^at_2^bt_3^cg_1r_{1,2})  =  \U {\rm Tr}(t_1^at_2^bt_3^cr_{1,2})
  \end{align*}

Next, suppose that $m\in \mathsf{C}_3$. This means that $m=t_1^at_2^bt_3^cwm_3$ for some $w\in \mathtt{C}_2$ and $m_3\in
\{g_2g_1b_1, g_2b_2, b_3, g_2g_1, g_2, 1\}$. From the previous results and from Lemma~\ref{multfraA} we obtain the following:

\begin{align*}
  &{\rm Tr}(t_1^at_2^bt_3^cwg_2b_2r_{1,2})=\U^{-1}{\rm Tr}(t_1^at_2^bt_3^cwg_2g_1b_1r_{1,2})\\
  &{\rm Tr}(t_1^at_2^bt_3^cwg_2g_1r_{1,2})=\U^2 {\rm Tr}(t_1^at_2^bt_3^cwr_{1,2})\\
  & {\rm Tr}(t_1^at_2^bt_3^cwb_3r_{1,2})=\U^{-2}{\rm Tr}(t_1^at_2^bt_3^cwg_2g_1b_1r_{1,2})\\
  & {\rm Tr}(t_1^at_2^bt_3^cwg_2r_{1,2})=\U {\rm Tr}(t_1^at_2^bt_3^cwr_{1,2})
\end{align*}

Therefore, we only have to study  the term ${\rm Tr}(t_1^at_2^bt_3^cwg_2g_1b_1r_{1,2})$. Replacing $w$ for  each element in
$\mathtt{C}_2$ and by using previous lemmas and results, we obtain the following:

\begin{align*}
 &{\rm Tr}(t_1^at_2^bt_3^cg_2g_1b_1r_{1,2})= \U^2 y {\rm Tr}(t_1^at_2^bt_3^cb_1r_{1,2})\\
 & {\rm Tr}(t_1^at_2^bt_3^cb_2g_2g_1b_1r_{1,2})  =  \U {\rm Tr}(t_1^at_2^bt_3^cb_1g_1b_1r_{1,2})\\
 &{\rm Tr}(t_1^at_2^bt_3^cg_1g_2g_1b_1r_{1,2})  = \U^3 y {\rm Tr}(t_1^at_2^bt_3^cb_1r_{1,2})\\
 &{\rm Tr}(t_1^at_2^bt_3^cg_1b_1 g_2g_1b_1r_{1,2})  = \U^2 {\rm Tr}(t_1^at_2^bt_3^cb_1g_1b_1r_{1,2}) \\
 &{\rm Tr}(t_1^at_2^bt_3^cb_1g_2g_1b_1r_{1,2})  = \U {\rm Tr}(t_1^at_2^bt_3^cb_1g_1b_1r_{1,2})\\
 &{\rm Tr}(t_1^at_2^bt_3^cb_1g_1g_2g_1b_1r_{1,2}) =  \U^2 {\rm Tr}(t_1^at_2^bt_3^cb_1g_1b_1r_{1,2})\\
 &{\rm Tr}(t_1^at_2^bt_3^cb_1b_2g_2g_1b_1r_{1,2}) = \U^{-1}{\rm Tr}(t_1^at_2^bt_3^cb_1g_1b_1g_2g_1b_1r_{1,2})
\end{align*}

The result for $n=3$ follows immediately. Finally, we suppose that the argument holds for $w\in \mathtt{C}_{n-1}$, and let $m\in
\mathtt{C}_n$. We have that $m^{(a)}=t_1^{a_1}\dots t_n^{a_n}wb_{n}^a$ or $m^\prime=t_1^{a_1}\dots t_n^{a_n}wg_{n-1}\dots
g_ib_i^a$, with $w\in \Z$, $a=0,1$ and $0\leq a_1,\dots, a_n \leq d-1$. Since we have that
  \begin{eqnarray*}
    {\rm Tr}(m^{(0)})= {\rm Tr}(t_1^{a_1}\dots t_n^{a_n}w) &=& x_{a_n}{\rm Tr}(w) \\
   {\rm Tr}(m^{(1)})= {\rm Tr}(t_1^{a_1}\dots t_n^{a_n}wb_{n}) &=& y_{a_n}{\rm Tr}(w) \\
   {\rm Tr}(m^\prime)= {\rm Tr}(t_1^{a_1}\dots t_n^{a_n}wg_{n-1}\dots g_ib_i^a)  &=& z{\rm Tr}(\alpha),
     \end{eqnarray*}
where $\alpha=t_1^{a_1}\dots t_{n-1}^{a_{n-1}}wg_{n-2}\dots g_it_i^{a_n}b_i^a \in \Z$,  the result follows by the induction
hypothesis.
\end{proof}

From the above proposition it is clear that it would be useful to compute the traces of the following elements:
$t_1^at_2^bt_3^cb_1r_{1,2}$, $ t_1^at_2^bt_3^cb_1g_1b_1r_{1,2}$ and $t_1^at_2^bt_3^cb_1g_1b_1g_2g_1b_1r_{1,2}$. We shall treat each
case as a separate lemma. For the first term we have:
\[
t_1^at_2^bt_3^cb_1r_{1,2} = t_1^at_2^bt_3^cb_1e_1 e_2 g_{1,2} = e_1^{(m)} e_2 b_1 g_{1,2}.
\]

\begin{lemma}\label{traceb1r12}
The following holds in  ${\rm Y}^{\mathtt{B}}_{d,n}(\U, \V)$:
\[
{\rm Tr}(e_1^{(m)} e_2 b_1 g_{1,2}) =   \frac{1}{d^2} \sum_{s,r=0}^{d-1}x_{-r} x_{-s+r} y_{m+s} + (\U^2+2) \frac{\U z}{d}
\sum_{r=0}^{d-1} x_{-r} y_{m+r}  + (\U^2+1) \U^2 z^2 y_m.
\]
\end{lemma}
\begin{proof}
We start by expanding the term $g_{1,2}$.

\begin{align*}
{\rm Tr}(e_1^{(m)} e_2 b_1 g_{1,2}) =& {\rm Tr}(e_1^{(m)} e_2 b_1) + \U \left( {\rm Tr}(e_1^{(m)} e_2 b_1g_1) + {\rm Tr}(e_1^{(m)}
e_2 b_1 g_2)\right) \\
&+ \U^2 \left ( {\rm Tr}(e_1^{(m)} e_2 b_1g_1 g_2) + {\rm Tr}(e_1^{(m)} e_2 b_1 g_2 g_1) \right ) +\U^3 {\rm Tr}(e_1^{(m)} e_2 b_1g_1
g_2 g_1)\\
=& \frac{1}{d^2} \sum_{s,r=0}^{d-1}x_{-r} x_{-s+r} y_{m+s} + \U \frac{2z}{d} \sum_{r=0}^{d-1} x_{-r} y_{m+r} + 2 \U^2 z^2 y_m \\
&+\U^3 \frac{z}{d^2} \sum_{s,r=0}^{d-1} x_{m+s-r} y_{-s+r} + \U^3 ( \U- \U^{-1}) z^2 y_m\\
&=\frac{1}{d^2} \sum_{s,r=0}^{d-1}x_{-r} x_{-s+r} y_{m+s} + (\U^2+2) \frac{\U z}{d} \sum_{r=0}^{d-1} x_{-r} y_{m+r}  + (\U^2+1)
\U^2 z^2 y_m.
\end{align*}

\end{proof}
For the term $ t_1^at_2^bt_3^cb_1g_1b_1r_{1,2}$ we have that:
\[
t_1^at_2^bt_3^cb_1g_1b_1r_{1,2}= t_1^a t_2^{b}t_3^cb_1g_1b_1 e_1e_2g_{1,2} = e_1^{(m)}e_2b_1g_1b_1g_{1,2}.
\]
Denote now
$A:={\rm Tr}(e_1^{(m)}e_2b_1g_1b_1g_{1,2})$, where $m=a+b+c$. By expanding $g_{1,2}$ we obtain that
$$A= A_1 + \U(A_2+A_3)+\U^2(A_4+A_5)+\U^3 A_6.$$
with:
$$\begin{array}{cccccc}
   A_1 & = & {\rm Tr}(e_1^{(m)}e_2b_1g_1b_1) \quad&  A_4  & = & {\rm Tr}(e_1^{(m)}e_2b_1g_1b_1g_1g_2) \\
   A_2 & =  & {\rm Tr}(e_1^{(m)}e_2b_1g_1b_1g_1) \quad&   A_5&=  &  {\rm Tr}(e_1^{(m)}e_2b_1g_1b_1g_2g_1)\\
   A_3 & = & {\rm Tr}(e_1^{(m)}e_2b_1g_1b_1g_2) \quad&   A_6& = &   {\rm Tr}(e_1^{(m)}e_2b_1g_1b_1g_1g_2g_1)
\end{array}$$

For the term $t_1^at_2^bt_3^cb_1g_1b_1g_2g_1b_1r_{1,2}$ we work in an analogous way.
Denote the following
\begin{eqnarray*}
 B  &:=&{\rm Tr}(e_1^{(m)}e_2b_1g_1b_1g_2g_1b_1g_{1,2}) = {\rm Tr}(t_1^a t_2^{b}t_3^cb_1g_1b_1 g_2g_1b_1e_1e_2g_{1,2})  \\
   &=&{\rm Tr}(t_1^at_2^bt_3^cb_1g_1b_1g_2g_1b_1r_{1,2}),
\end{eqnarray*}
where $m=a+b+c$. By expanding the term $g_{1,2}$ we obtain:
$$B= B_1 + \U(B_2+B_3)+\U^2(B_4+B_5)+\U^3 B_6$$
with
$$\begin{array}{cccccc}
  B_1 & = & {\rm Tr}(e_1^{(m)}e_2b_1g_1b_1g_2g_1b_1) \quad& B_4 & = & {\rm Tr}(e_1^{(m)}e_2b_1g_1b_1g_2g_1b_1g_1g_2) \\
  B_2 & = & {\rm Tr}(e_1^{(m)}e_2b_1g_1b_1g_2g_1b_1g_1)  \quad& B_5 & = &  {\rm Tr}(e_1^{(m)}e_2b_1g_1b_1g_2g_1b_1g_2g_1)\\
  B_3 & = & {\rm Tr}(e_1^{(m)}e_2b_1g_1b_1g_2g_1b_1g_2) \quad&  B_6& = &{\rm Tr}(e_1^{(m)}e_2b_1g_1b_1g_2g_1b_1g_1g_2g_1)
\end{array}$$
We then have the following lemma:

\begin{lemma}\label{traceABis}
The following holds in ${\rm Y}^{\mathtt{B}}_{d,n}(\U, \V)$:

\begin{align*}
A=& \frac{z}{d} \sum_{r=0}^{d-1} x_{-r} x_{m+r}+(\V+\V^{-1})\frac{z}{d} \sum_{r=0}^{d-1} x_{-r} y_{m+r} +
\frac{1}{d^2}\sum_{r,s}x_{-r}y_{m+s}y_{r-s}+(\U-\U^{-1})A_1 \nonumber \\
&+z^{2}x_m+ (\V - \V^{-1})  \frac{z^2}{d}\sum_r y_{r} +  2\left ( \frac{z}{d} \sum_s y_{m+s}y_{-s} +(\U-\U^{-1})A_3 \right) +
A_3+(\U-\U^{-1})A_4.\\
B=& \frac{z}{d} \sum_{r=0}^{d-1} x_{-r} y_{m+r}+(\U-\U^{-1})z^2y_m+ (\V-\V^{-1})\frac{z}{d^2}\sum_{k,r}y_{-k}y_{m+k+r}\\
&+ (\V-\V^{-1})(\U-\U^{-1})\frac{z^2}{d}\sum_rx_{m+r} +(\V-\V^{-1})^2(\U-\U^{-1})\frac{z^2}{d}\sum_r y_{m+r}\\
& + 2 \left (z^2y_m+z^2(\V-\V^{-1})\frac{1}{d}\sum_rx_{r}  +z^2(\V-\V^{-1})^2\frac{1}{d}\sum_ry_{r}+(\U-\U^{-1})B_1 \right )\\
& + 2( B_1+(\U-\U^{-1})B_2) + \frac{1}{d^2} \sum_{s,k} y_{-k} y_{-s+k} y_{m+s} + (\U - \U^{-1}) \frac{z}{d} \\
& + \sum_{k} y_{-k} x_{m+k} + (\U - \U^{-1})(\V - \V^{-1}) \frac{z}{d^2} \sum_{r,k} y_{-k} y_{m+r+k} + (\U - \U^{-1}) ( B_1 +
B_5).\\
\end{align*}
\end{lemma}

\begin{proof}
 The proof is a long straightforward computation. For instance, for the expressions $A_1, A_2$ and $B_1$ we have:

\begin{align*}
 A_1  &=  \frac{1}{d^2}\sum_{r,s}{\rm Tr}(t_1^{m+s}t_2^{r-s}t_3^{-r}b_1g_1b_1)=  \frac{z}{d}\sum_{s}{\rm Tr}(t_1^{m+s}t_2^{-s}b_1^2)
 \\
   &=  \frac{z}{d}\sum_{s}[{\rm Tr}(t_1^{m+s}t_2^{-s}) + (\V-\V^{-1}){\rm Tr}(t_1^{m+s}t_2^{-s}b_1f_1)]\\
   &= \frac{z}{d}\sum_{s}[{\rm Tr}(t_1^{m+s}t_2^{-s}) + (\V-\V^{-1})\frac{1}{d}\sum_r{\rm Tr}(t_1^{m+s+r}t_2^{-s}b_1)]  \\
   &= \frac{z}{d}\sum_{s}[x_{m+s}x_{-s} + (\V-\V^{-1})\frac{1}{d^2}\sum_r x_{-s}y_{m+s+r}]\\
A_3 &= \frac{1}{d^2}\sum_{r,s}{\rm Tr}(t_1^{m+s}t_2^{r-s}t_3^{-r}b_1g_1b_1g_2) =  \frac{z}{d}\sum_{k}{\rm
Tr}(t_1^{m+k}t_2^{-k}b_1g_1b_1)= z^2{\rm Tr}(t_1^{m}b_1^2)  \\
    &= z^2[{\rm Tr}(t_1^{m})+(\V-\V^{-1}){\rm Tr}(t_1^{m}b_1f_1)]= z^{2}x_m+ (\V - \V^{-1}) \frac{z^2}{d}\sum_r y_{m+r}\\
    &=z^{2}x_m+ (\V - \V^{-1})  \frac{z^2}{d}\sum_r y_{r}.
\end{align*}

\begin{align*}
B_1 &=  \frac{1}{d^2}\sum_{r,s}{\rm Tr}(t_1^{m+s}t_2^{r-s}t_3^{-r}b_1g_1b_1g_2g_1b_1)=  \frac{1}{d^2}\sum_{r,s}{\rm
Tr}(t_1^{m+s}t_2^{r-s}b_1g_1b_1g_2 t_2^{-r}g_1b_1)\\
   &=   \frac{z}{d^2}\sum_{r,s}{\rm Tr}(t_1^{m+s-r}t_2^{r-s}b_1g_1b_1g_1b_1)=  \frac{z}{d}\sum_{k}{\rm
   Tr}(t_1^{m+k}t_2^{-k}b_1^2g_1b_1g_1)
\end{align*}

\begin{align*}
   &=  \frac{z}{d}\sum_{k}[{\rm Tr}(t_1^{m+k}t_2^{-k}g_1b_1g_1)+(\V-\V^{-1}){\rm Tr}(t_1^{m+k}t_2^{-k}b_1f_1g_1b_1g_1)]\\
  &=  \frac{z}{d}\sum_{k}[{\rm Tr}(t_1^{m+k}t_2^{-k}g_1b_1g_1^{-1})+(\U-\U^{-1}){\rm Tr}(t_1^{m+k}t_2^{-k}g_1b_1)+\\
   &(\V-\V^{-1}){\rm Tr}(t_1^{m+k}t_2^{-k}f_1b_1g_1b_1g_1^{-1})+(\V-\V^{-1})(\U-\U^{-1}){\rm Tr}(t_1^{m+k}t_2^{-k}f_1b_1g_1b_1)] \\
   &=  \frac{z}{d}\sum_{k}[y_{-k}x_{m+k}+(\U-\U^{-1})zy_m+(\V-\V^{-1})\frac{1}{d}\sum_ry_{-k}y_{m+k+r}\\
   &+z(\V-\V^{-1})(\U-\U^{-1}){\rm Tr}(t_1^{m}f_1b_1^2)] \\
   &=\frac{z}{d}\sum_{k}y_{-k}x_{m+k}+(\U-\U^{-1})z^2y_m+ (\V-\V^{-1})\frac{z}{d^2}\sum_{k,r}y_{-k}y_{m+k+r}\\
   &+ (\V-\V^{-1})(\U-\U^{-1})\frac{z^2}{d}\sum_rx_{m+r} + (\V-\V^{-1})^2(\U-\U^{-1})\frac{z^2}{d}\sum_r y_{m+r}.
\end{align*}

In a similar way, we obtain the following equations for the remaining expressions.
  \begin{align*}
  A_2 &=\frac{1}{d^2}\sum_{r,s}x_{-r}y_{m+s}y_{r-s}+(\U-\U^{-1})A_1.\\
  A_4 &=\frac{z}{d} \sum_s y_{m+s}y_{-s}  +(\U-\U^{-1})A_3, \quad A_5 = A_4. \\
  A_6 &= A_3+(\U-\U^{-1})A_4.
\end{align*}
and
$$\begin{array}{rlrl}
 B_2 &=z^2y_m+z^2(\V-\V^{-1})\frac{1}{d}\sum_rx_{r}+z^2(\V-\V^{-1})^2\frac{1}{d}\sum_ry_{r}+(\U-\U^{-1})B_1, & &\\
  B_4 &= \frac{z}{d} \sum_{r} {\rm Tr} (t_1^{m+r} t_2^{-r} b_1^2 g_1 b_1 g_1) + \frac{z}{d} (\U - \U^{-1}) \sum_{r} {\rm
  Tr}(t_1^{m+r}
t_2^{-r} b_1 g_1 b_1 g_1 b_1 g_1 ).&  & \\
 B_3 &=B_2,\quad  B_5 = B_4. & &\\
B_6 &= \frac{1}{d^2} \sum_{s,k} y_{-k} y_{-s+k} y_{m+s} + (\U - \U^{-1}) \frac{z}{d} \sum_{k} y_{-k} x_{m+k} & &\\
  &  + (\U - \U^{-1})(\V -  \V^{-1})\frac{z}{d^2} \sum_{r,k} y_{-k} y_{m+r+k}  + (\U - \U^{-1}) ( B_1 + B_5), & &
\end{array}$$
which implies the result.
\end{proof}

\subsection{{\it A Markov trace on the algebra ${\rm FTL}^\mathtt{B}_{d,n}$}}
In order to find the necessary and sufficient conditions so that ${\rm Tr}$  passes
to ${\rm FTL}_{d,n}^{\mathtt{B}}(\U,\V)$, one has to make sure that ${\rm Tr}$ annihilates the defining ideal
$\langle r_{1,2} , r_{\mathtt{B}} \rangle$ of ${\rm FTL}_{d,n}^{\mathtt{B}}(\U,\V)$. For this reason, we have to solve the
following system of equations:

\begin{equation}\label{condsystem} \left (\Sigma \right) =  \left\{\begin{array}{r}
A=0\\
B=0\\
{\rm Tr}(r_{\mathtt{B}}) =0 \\
{\rm Tr} (e_1^{(m)}e_2 r_{1,2}) =0 \\
{\rm Tr}(e_1^{(m)}e_2 b_1 r_{1,2}) =0
\end{array}\right.
\end{equation}

The above system may initially seem intimidating, however, using harmonic analysis on the underlying finite group simplifies
things considerably. We shall follow the method of P. Gerard\'in \cite[Appendix]{jula2}. We will first write the above system in its functional notation and then apply the Fourier transform, which is a standard tool in the theory of framization of knot algebras \cite{gojukola,
gojukola2, fjl}.
We shall treat separately the first two equations because of
their length.

Before solving \eqref{condsystem}, we will make a short digression on the  Fourier transform of a
complex function on a finite cyclic group. Let $L(C_d):=\mathbb{C}[C_d]$ be the group algebra formed by all complex functions on
$C_d$. The convolution product in this algebra is defined by:
\[
(f \ast g)(x) = \sum_{y\in C_d} f(y) g(x-y) \quad \mbox{where} f,g \in \mathbb{C} [C_d].
\]
We also define the product by coordinates in $L(C_d)$ as follows:
\[
fg : x \rightarrow f(x)g(x)\quad \mbox{where} f, g \in C_d.
\]
The set $\{ \delta_a \ | \ a \in  C_d \}$, where $\delta_a \in L(C_d)$ is the
function with support $\{ a \}$, is a linear basis for $ L(C_d)$ with respect to the convolution
product. From now on we will consider $C_d$ as an additive group, that is, $C_d=\mathbb{Z}/d\mathbb{Z}$. The Fourier transform $\mathcal{F}$ is the linear automorphism on $L(C_d)$ defined by $f\mapsto \widehat{f}$, with
\[
\widehat{f}(k) : = ( f \ast \chi_k)(0) = \sum_{y\in C_d}f(y) \chi_k(-y),
\]
where $\chi_k:  a \mapsto \cos \left ( \frac{2 \pi k a} {d} \right ) + i \sin \left ( \frac{2 \pi k a} {d} \right )$  denote the characters of $C_d$ for $k\in
\{0,\ldots, d-1\}$. Note that $\Hat{\Hat{f}}(x) = d
f(-x)$. Finally, note that the
elements in the group algebra $L(C_d)$ can also be identified with the set of formal sums $\{\sum_{s=0}^{d-1} \alpha_s t^s\
|\alpha_s\in \mathbb{C}\}$ as follows:
\[
(f :  C_d \rightarrow \mathbb{C}) \longleftrightarrow \sum_{s=0}^{d-1} f(t^s)t^s
\]
We will often use this identification, since it makes some computations easier. For details regarding the properties of the
convolution product and the Fourier transform the reader is referred to \cite{te, jula, gojukola2,fjl}.\\

We are now ready to solve \eqref{condsystem}. We start with equation $A=0$. Denote its functional form by $\mathfrak{F}A=0$  and
consider the function $\mathbf{1}:C_d\rightarrow \mathbb{C}$ defined by $\mathbf{1}(m)=1$ for all
$m\in C_d $. We then have:
\[
\mathfrak{F}A= \mathfrak{F}A_1 + \U(\mathfrak{F}A_2+\mathfrak{F}A_3)+\U^2(\mathfrak{F}A_4+\mathfrak{F}A_5)+\U^3 \mathfrak{F}A_6,
\]
where
\begin{align*}
\mathfrak{F}A_1 &= \frac{z}{d}x*x+(\V-\V^{-1})\frac{z}{d^2}x*y*\mathbf{1};\quad \mathfrak{F}A_2 =
\frac{1}{d^2}x*y*y+(\U-\U^{-1})A_1\\
\mathfrak{F}A_3 &= z^2x+(\V-\V^{-1})\frac{z^2}{d}y*\mathbf{1};\quad \mathfrak{F}A_4 =  \frac{z}{d}y*y+(\U-\U^{-1})A_3 =
\mathfrak{F}A_5\\
 \mathfrak{F}A_6 &= A_3+(\U-\U^{-1})A_4
 \end{align*}

The case of equation $B=0$ is analogous.
 \[
\mathfrak{F} B= \mathfrak{F}B_1 + \U(\mathfrak{F}B_2+\mathfrak{F}B_3)+\U^2(\mathfrak{F}B_4+\mathfrak{F}B_5)+\U^3 \mathfrak{F}B_6,
 \]
 where:

 \begin{align*}
\mathfrak{F}B_1 &=\frac{z}{d}x*y+(\U-\U^{-1})z^2
y+(\V-\V^{-1})\frac{z}{d^2}y*y*\mathbf{1}+\frac{z^2}{d}(\V-\V^{-1})(\U-\U^{-1})x*\mathbf{1}\\
   &+\frac{z^2}{d}(\V-\V^{-1})^2(\U-\U^{-1})y*\mathbf{1}\\
\mathfrak{F}B_2 &= z^2 y+\frac{z^2}{d}(\V-\V^{-1})x*\mathbf{1}+\frac{z^2}{d}(\V-\V^{-1})^2y*\mathbf{1}+ (\U-\U^{-1})B_1 =
\mathfrak{F}B_3\\
\mathfrak{F}B_4 &= B_1+(\U-\U^{-1})B_2 = \mathfrak{F}B_5\\
 \mathfrak{F}B_6 &= \frac{1}{d^2}y*y*y+\frac{z}{d}(\U-\U^{-1})x*y+(\U-\U^{-1})(\V-\V^{-1})\frac{z}{d^2}y*y*\mathbf{1}+(\U -
 \U^{-1}) ( B_1 + B_5).
 \end{align*}

From the above, the system $\left ( \Sigma \right )$  becomes:
{\small
\begin{align}
&\mathfrak{F}A=0 \label{sys1}\\
&\mathfrak{F}B=0 \label{sys2}\\
&x*(x*\mathbf{1})+\U^2\V^2y*(y*\mathbf{1})+\V(\U^2+1)x*(y*\mathbf{1}) + \nonumber\\
&\hspace{1.4cm}+ dz\U(1+\U^2\V^2)x*\mathbf{1}+ dz(\U^3\V^3 + \U
\V)y*\mathbf{1} =0 \label{sys3}\\
 &x*(x*x)+dz\U(\U+2)x*x+d^2z\U^2(\U^2+1)x =0 \label{sys4}\\
 &x*(x*y)+dz\U(\U+2)x*y+d^2z\U^2(\U^2+1)y =0 \label{sys5}
 \end{align}
}
Let $x_0,\ldots,x_{d-1}$ and $y_0,\ldots,y_{d-1}$ be the parameters of {\rm
Tr}. Let also $x:C_d \rightarrow
\mathbb{C}$ the function such that  $x(0)=1$ and $x(k)=x_k$, $1 \leq k \leq d-1$, and let  $y:C_d
\rightarrow \mathbb{C}$ be the function such that $y(k)=y_k$, $0 \leq k \leq d-1$.\smallbreak

We will solve the system of equations \eqref{sys1}-\eqref{sys5}. We start with \eqref{sys4}, apply the Fourier
transform, and reproduce the proof of  \cite[Theorem~6 and Section~7]{gojukola2}. We obtain the following values for
$\widehat{x}$:

\begin{equation}\label{FTxsol}
\widehat{x} = - \left (d\U z \sum_{m\in {\rm Sup}_1}t^m + d\U (\U^2+1) z \sum_{m\in {\rm Sup}_2}t^m \right ).
\end{equation}

Using the properties of the Fourier transform, we obtain the expression for the $x_k$'s:
\begin{equation}\label{xsol}
x_k = - z \left (\U  \sum_{m \in {\rm Sup}_1} \chi_m(k) + \U (\U^2+1) \sum_{m \in {\rm Sup}_2} \chi_m (k) \right ).
\end{equation}

Next, we use \eqref{sys5} to detemine ${\rm Sup}(\widehat{y})$. By applying the Fourier transform once again we obtain:
$$\underbrace{(\widehat{x}^2+dz\U(\U+2)\widehat{x}+d^2z\U^2(\U^2+1))}_{D}\widehat{y}=0. $$
We know that $D=0$ for all $m\in {\rm Sup}(\widehat{x})$, therefore, $\widehat{y}$ can be free in ${\rm
Sup}(\widehat{x})$. On the other hand, if $n\not\in {\rm Sup}(\widehat{x})$ we obtain:
$$d^2z\U^2(\U^2+1)\widehat{y}(n)=0,$$
which implies that $\widehat{y}(n)=0$ and therefore, supposing that
$d^2z\U^2(\U^2+1)\not=0$), we deduce that ${\rm Sup}(\widehat{y})\subseteq {\rm Sup}(\widehat{x})$.

We will use the expressions for  $x$ and for ${\rm Sup} ( \widehat{y} ) $ to solve the
remaining equations.
Let $\mathbf{1} = \sum_{k=0}^{d-1} \mathbf{1}(m) t^m$ and observe that the Fourier transform of the function $\mathbf{1} :
C_d \rightarrow \mathbb{C}$ is:
\[
\widehat{\mathbf{1}} = \sum_{s=0}^{d-1} (\mathbf{1} \ast \mathbf{i}_s (0)) t^s = \sum_{s=0}^{d-1} \left [
\sum_{r=0}^{d-1}\mathbf{1}(r) \chi_s(-r)  \right] t^s  = \sum_{s=0}^{d-1} \left [  \sum_{r=0}^{d-1} \chi_s(-r)  \right]
t^s.
\]
 Thus we have that:
 \[
 \widehat{\mathbf{1}}(k) = \left \{ \begin{array}{ll} d, & k=0 \\ 0, & k \neq 0  \end{array} \right. .
 \]
This means that in order to obtain the full set of solutions for $(\Sigma$) we will have to solve  \eqref{sys1}-\eqref{sys3} for
both zero and non-zero values of $k$.
Moreover, from \eqref{FTxsol} and depending on which subset of ${\rm Sup}(\widehat{x})$ the element $k$ lies in,   we have the following
possibilities for $\widehat{x}(k)$:
 \[
\widehat{x}(k) = \left \{ \begin{array}{ll} -d \U z , & k \in {\rm Sup}_1 \\ -d \U ( \U^2+1) z, & k \in {\rm Sup}_2 \end{array}
\right. , \ k \in C_d  .
\]
For $k\in C_d$ and $k \neq 0$, the equation \eqref{sys3} vanishes and we obtain the following solutions for $\widehat{y}(k)$:

\begin{equation}\label{knon0}
\widehat{y}(k)=\left\{\begin{array}{ll}
  - d \U z \quad\hbox{or}\quad   d \U z, & \hbox{if $k \in{\rm Sup}_1$}, \ k\neq0\\
  0\quad \hbox{or} \quad \quad - d \U z ( \U^2 +1)\quad \hbox{or} \quad  d \U z (\U^2+1), & \hbox{if $k \in{\rm Sup}_2$}, \ k\neq 0
\end{array}\right. .
\end{equation}

On the other hand, for $k=0$ we obtain the following values for $\widehat{y}(0)$:

\begin{equation}\label{k=0}
\widehat{y}(0)=\left\{\begin{array}{ll}
   \frac{d \U z}{\V} \quad\hbox{or}\quad  - d \U\V  z, & \hbox{if $0 \in{\rm Sup}_1$}\\
   \frac{d \U z ( \U^2 +1)}{\V}\quad \hbox{or} \quad  \frac{d \U z (1-\V^2)}{\V}, & \hbox{if $0 \in{\rm Sup}_2$}
\end{array}\right. .
\end{equation}

Combining \eqref{knon0} and \eqref{k=0}, we deduce the following four solutions for $\widehat{y}$:
{\small
\begin{align*}
   \widehat{y}^1&= -d\U z\left( -\frac{1}{\V}+ \sum_{m\in{\rm Sup}^y_1}t^m - \sum_{m\in{\rm Sup}^y_2} t^m + (\U^2+1) \sum_{m\in{\rm
   Sup}^y_3} t^m - (\U^2+1) \sum_{m\in{\rm Sup}^y_4} t^m \right)\\
 \widehat{y}^2&= -d\U z\left( \V+ \sum_{m\in{\rm Sup}^y_1}t^m - \sum_{m\in{\rm Sup}^y_2} t^m + (\U^2+1) \sum_{m\in{\rm Sup}^y_3}
 t^m - (\U^2+1) \sum_{m\in{\rm Sup}^y_4} t^m \right)\\
 \widehat{y}^3&= -d\U z\left( -\frac{(\U^2+1)}\V+ \sum_{m\in{\rm Sup}^y_1}t^m - \sum_{m\in{\rm Sup}^y_2} t^m + (\U^2+1)
 \sum_{m\in{\rm Sup}^y_3} t^m - (\U^2+1) \sum_{m\in{\rm Sup}^y_4} t^m \right)\\
  \widehat{y}^4&= -d\U z\left( \frac{\V^2-1}{\V}+ \sum_{m\in{\rm Sup}^y_1}t^m - \sum_{m\in{\rm Sup}^y_2} t^m + (\U^2+1)
  \sum_{m\in{\rm Sup}^y_3} t^m - (\U^2+1) \sum_{m\in{\rm Sup}^y_4} t^m \right),
\end{align*}
}
where ${\rm Sup}^y_i=\{a\in C_d \ |\ \widehat{y}(a)=f(i)\ \mbox{for} \ 0 \leq i \leq 4 \}$ and $f:\{1,\dots, 4\}\rightarrow
\mathbb{C}$ is the function that is defined by  $f=-d\U z\delta_1+ d\U z \delta_2- d\U z(\U^2+1)\delta_3+ d\U z(\U^2+1)\delta_4$.
Moreover, from the above definitions for ${\rm Sup}_i^y$ together with  \eqref{knon0} and \eqref{k=0} we deduce the inclusions:
\begin{align*}
&{\rm Sup}^y_1\sqcup {\rm Sup}^y_2\sqcup\{0\} = {\rm Sup}_1\  \mbox{and} \  {\rm Sup}^y_3\sqcup{\rm Sup}^y_4\subseteq {\rm Sup_2},
\ \mbox{if} \  0\in {\rm Sup}_1. \\
& {\rm Sup}^y_1\sqcup {\rm Sup}^y_2 = {\rm Sup}_1 \  \mbox{and}  \ {\rm Sup}^y_3\sqcup{\rm Sup}^y_4\sqcup\{0\}\subseteq {\rm
Sup_2}, \ \mbox{if} \ 0\in {\rm Sup}_2.
\end{align*}

Using now the properties of the Fourier transform we are able to determine the expression for $y_k^{r}$'s,  $k\in
C_d$ and $r\in \{1,\dots 4\}$:

\begin{align*}
   y^1_k&= -\U z\left( -\frac{1}{\V}\chi_0(k)+ \sum_{m\in{\rm Sup}^y_1}\chi_m(k) - \sum_{m\in{\rm Sup}^y_2}\chi_m(k) +
   (\U^2+1) \sum_{m\in{\rm Sup}^y_3} \chi_m(k)- (\U^2+1) \sum_{m\in{\rm Sup}^y_4}\chi_m(k)\right)\\
 y^2_k&= -\U z\left( \V\chi_0(k)+ \sum_{m\in{\rm Sup}^y_1}\chi_m(k) - \sum_{m\in{\rm Sup}^y_2}\chi_m(k) + (\U^2+1)
 \sum_{m\in{\rm Sup}^y_3} \chi_m(k) - (\U^2+1) \sum_{m\in{\rm Sup}^y_4} \chi_m(k) \right)\\
 y^3_k&= -\U z\left( -\frac{(\U^2+1)}\V\chi_0(k)+ \sum_{m\in{\rm Sup}^y_1}\chi_m(k) - \sum_{m\in{\rm Sup}^y_2}\chi_m(k) +
 (\U^2+1) \sum_{m\in{\rm Sup}^y_3}\chi_m(k) - (\U^2+1) \sum_{m\in{\rm Sup}^y_4}\chi_m(k) \right)\\
 y^4_k&= -\U z\left( \frac{\V^2-1}{\V}\chi_0(k)+ \sum_{m\in{\rm Sup}^y_1}\chi_m(k) - \sum_{m\in{\rm Sup}^y_2} \chi_m(k) +
 (\U^2+1) \sum_{m\in{\rm Sup}^y_3} \chi_m(k) - (\U^2+1) \sum_{m\in{\rm Sup}^y_4}\chi_m(k)\right) .
\end{align*}

Finally, we return to \eqref{xsol} in order to determine the values of the trace parameter $z$. Recall that $x_0=1$ and thus we
have:
\begin{equation}\label{zvalsol}
1=x_0 = - z \left ( \U | {\rm Sup}_1 | + \U (\U^2+1)|{\rm Sup}_2| \right),
\end{equation}

or, equivalenlty:

\[
z= - \frac{1}{\U | {\rm Sup}_1 | + \U (\U^2+1)|{\rm Sup}_2|}.
\]
We thus have proven the main theorem of this paper, which is the following:

\begin{theorem}\label{mainthm}
Let $x: C_d \rightarrow \mathbb{C}$ such that $x(0)=1$ and $x(k) = x_k$, $1 \leq k \leq d-1$ and let also $y:
C_d \rightarrow \mathbb{C}$ such that $y(k) = y_k$, $0 \leq k \leq d-1$. The trace ${\rm Tr}$ defined on ${\rm
Y}_{d,n}^{\mathtt{B}}(\U,\V)$ passes to the quotient algebra ${\rm FTL}_{d,n}^{\mathtt{B}}(\U,\V)$ if and only if the parameters of
the trace satisfy the following conditions:
$$
x_k = - z \left ( \U \sum_{m \in {\rm Sup}_1} \chi_m(k) + \U  (\U^2+1)  \sum_{m \in {\rm Sup}_2} \chi_m (k) \right ),\hbox{and}
$$
$$ z= - \frac{1}{ \U| {\rm Sup}_1 | + \U (\U^2+1) |{\rm Sup}_2|},
$$
where ${\rm Sup}_1 \sqcup {\rm Sup_2}$ is the support of the Fourier transform of $x$, $\widehat{x}$. Moreover, we have that:
\[
{\rm Sup}(\widehat{y}) \subseteq {\rm Sup}(\widehat{x}),
\]
where $\widehat{y}$ is the Fourier transform of $y$ and one of the two cases holds:\bigskip

\noindent 1.  If $0\in {\rm Sup}_1$, the parameters $y_k$ have the following form:

\begin{align*}
  y_k=& -\U z\left ( - \frac{1}{\V}\chi_0(k)+ \sum_{m\in{\rm Sup}^y_1} \chi_m(k) - \sum_{m\in{\rm Sup}^y_2} \chi_m(k) +
  (\U^2+1) \sum_{m\in{\rm Sup}^y_3} \chi_m(k) - (\U^2+1) \sum_{m\in{\rm Sup}^y_4} \chi_m(k) \right )\\
  y_k=& -\U z\left (\V \chi_0(k)+ \sum_{m\in{\rm Sup}^y_1} \chi_m(k) - \sum_{m\in{\rm Sup}^y_2} \chi_m(k) +
   (\U^2+1) \sum_{m\in{\rm Sup}^y_3} \chi_m(k) - (\U^2+1) \sum_{m\in{\rm Sup}^y_4} \chi_m(k) \right ).
\end{align*}

\noindent 2. If $0\in {\rm Sup}_2$, the parameters $y_k$ have the following form:

\begin{align*}
  y_k=& -\U z\left ( -\frac{(\U^2+1)}{\V}\chi_0(k)+ \sum_{m\in{\rm Sup}^y_1} \chi_m(k) - \sum_{m\in{\rm Sup}^y_2} \chi_m(k)
  + (\U^2+1) \sum_{m\in{\rm Sup}^y_3} \chi_m(k) - (\U^2+1) \sum_{m\in{\rm Sup}^y_4} \chi_m(k) \right )\\
   y_k=& -\U z\left ( \frac{\V^2-1}{\V}\chi_0(k)+ \sum_{m\in{\rm Sup}^y_1} \chi_m(k) - \sum_{m\in{\rm Sup}^y_2} \chi_m(k) +
   (\U^2+1) \sum_{m\in{\rm Sup}^y_3} \chi_m(k) - (\U^2+1) \sum_{m\in{\rm Sup}^y_4} \chi_m(k) \right ).
 \end{align*}

where $\sqcup_{i=0}^{4}{\rm Sup}^y_i = {\rm Sup}(\widehat{y})$. Finally, the following holds:

\begin{align*}
&{\rm Sup}^y_1\sqcup {\rm Sup}^y_2\sqcup\{0\} = {\rm Sup}_1\  \mbox{and} \  {\rm Sup}^y_3\sqcup{\rm Sup}^y_4\subseteq {\rm Sup_2},
\ \mbox{if} \  0\in {\rm Sup}_1 \\
& {\rm Sup}^y_1\sqcup {\rm Sup}^y_2 = {\rm Sup}_1 \  \mbox{and}  \ {\rm Sup}^y_3\sqcup{\rm Sup}^y_4\sqcup\{0\}\subseteq {\rm
Sup_2}, \ \mbox{if} \ 0\in {\rm Sup}_2.
\end{align*}

\end{theorem}

\begin{corollary}\label{parvals}
In the case where one of ${\rm Sup}_1$ or ${\rm Sup}_2$ is the empty set, the values of the $x_k$'s are solutions of the ${\rm
E}$-system, while the the $y_k$'s are solutions of the ${\rm F}$-system. More precisely we have that:\bigskip

\noindent 1. If ${\rm Sup}_1 = \emptyset$, then:
\[
0 \in {\rm Sup}_2 , \quad x_k = \frac{1}{| {\rm Sup_2} |} \sum_{m \in {\rm Sup}_2} \chi_m (k), \quad z = -\frac{1}{\U (\U^2+1)
|{\rm Sup}_2|}
\]
and the $y_k$'s are one of the following solutions of the ${\rm F}$-system:

\begin{align*}
(i)\ y_k&= - \frac{1}{\V |{\rm Sup}_2 | } \chi_0(k) + \frac{1}{|{\rm Sup}_2 |} \left ( \sum_{m\in{\rm Sup}^y_3} \chi_m(k) -
\sum_{m\in{\rm Sup}^y_4} \chi_m(k) \right )\\
or & \\
(ii) \ y_k& =   \frac{\V^2 -1}{\V (\U^2+1)|{\rm Sup}_2 | } \chi_0(k) + \frac{1}{|{\rm Sup}_2 |} \left ( \sum_{m\in{\rm Sup}^y_3}
\chi_m(k) - \sum_{m\in{\rm Sup}^y_4} \chi_m(k) \right ).
\end{align*}

\noindent 2. If ${\rm Sup}_2 = \emptyset$, then:
\[
0 \in {\rm Sup}_1, \quad x_k = \frac{1}{| {\rm Sup_1} |}  \sum_{m \in {\rm Sup}_1} \chi_m (k), \quad z = -\frac{1}{\U  |{\rm
Sup}_1|}
\]
and the $y_k$'s are one of the following solutions of the ${\rm F}$-system:

\begin{align*}
(i) \ y_k&=  - \frac{1}{\V |{\rm Sup}_1 | } \chi_0(k) + \frac{1}{|{\rm Sup}_1 |}\left ( \sum_{m\in{\rm Sup}^y_1} \chi_m(k) -
\sum_{m\in{\rm Sup}^y_2} \chi_m(k)  \right ) \\
or & \\
(ii) \ y_k&=  \frac{\V }{|{\rm Sup}_1 | } \chi_0(k) +\frac{1}{|{\rm Sup}_1 |} \left ( \sum_{m\in{\rm Sup}^y_1} \chi_m(k) -
\sum_{m\in{\rm Sup}^y_2} \chi_m(k) \right ).
\end{align*}

\end{corollary}

\begin{remark}\rm
The conditions for the trace parameters $z$ and $x_m$, $0 \leq m \leq d-1$, are in total agreement with the corresponding necessary
and sufficient conditions for the type \texttt{A} case  \cite[Theorem~6 and Section~7]{gojukola2}. This is something that is
expected since classical knot theory embeds in the knot theory of the solid torus. Further, for $d=1$ these conditions are also
coherent with the solutions found for the classical case in Section~\ref{traceHB}.
\end{remark}

\section{Link Invariants from ${\rm FTL}^{\mathtt{B}}_{d,n}(\U , \V)$}\label{secinv}

In this section we introduce the framed and classical link invariants that are derived from ${\rm
FTL}_{d,n}^{\mathtt{B}}(\U , \V)$. In analogy to the type $\mathtt{A}$  case
\cite{gojukola2}, these invariants will be specializations of the invariants ${\mathcal X}_S^{\mathtt{B}}$, where $S \subset
C_d$, that were constructed on the level of $\Y (\U , \V)$ in \cite{fjl}. We shall first discuss briefly the
invariants ${\mathcal X}_S^{\mathtt{B}}$ and then we will proceed with the specialization.

\subsection{Invariants for framed links in the solid torus}
The closure of a framed or classical braid of type $\mathtt{B}$ corresponds to a knot or a link in the solid torus. Therefore, as
mentioned earlier, in order to define link invariants on the level of $\Y$, one has to make sure that the Markov trace ${\rm Tr}$
satisfies the Markov equivalence for modular framed braids in the solid torus. To be more precise, two elements in
$\bigcup_{n}\mathcal{F}^{\mathtt{B}}_{d,n}$ are equivalent if and only if  they differ by a finite sequence of conjugations  in the
groups $\mathcal{F}^{\mathtt{B}}_{d,n}$ and stabilization moves $\mathcal{F}^{\mathtt{B}}_{d,n} \ \ni \alpha \sim \alpha
\sigma_{n}^{\pm 1} \in \mathcal{F}^{\mathtt{B}}_{d,n+1}$. Let ${\rm X}= ({\rm x}_1 , \ldots , {\rm x}_{d-1})$ a solution of the
${\rm E}$-system, ${\rm Y} =( {\rm y}_0 , \ldots , {\rm y}_{d-1} ) $ a solution of the ${\rm F}$-system and $S \subset
C_d$ that parametrizes said solutions. Then ${\rm Tr}$ can be rescaled and normalized as follows:
\begin{definition}\label{xinv}
The following map is an invariant of framed links inside the solid torus:
\[
{\mathcal X}_S^{\mathtt{B}}(\lambda, \U, \V) (\widehat{\alpha})= \left ( \frac{1 - \lambda_S}{\sqrt{\lambda_S} ( \U - \U^{-1}) {\rm
E}_S} \right )^{n-1} \left ( \sqrt{\lambda_S} \right )^{\varepsilon(\alpha)}  {\rm Tr}(\pi (\alpha)),
\]
where $\lambda_S = \frac{z- ( \U - \U^{-1}) {\rm E}_S}{z}$ is the rescaling factor, ${\rm E_S} = \frac{1}{|S|}$ for all $i$
\cite{jula,jula4} , $\varepsilon(\alpha)$ is the algebraic sum of the exponents of the $\sigma_i$'s in $\alpha$ and $\pi$ is the
natural epimorphism $\pi: \mathcal{F}^{\mathtt{B}}_{d,n} \rightarrow \Y$. Restricting $\pi$ to classical braids, which can be seen
as framed braids with all framings zero, one obtains an invariant for classical links ${\mathcal Y}_S^{\mathtt{B}}(\lambda, \U, \V)
(\widehat{\alpha})$.
\end{definition}

In analogy to the classical case, we can prove that the invariants ${\mathcal X}_S^{\mathtt{B}}$ satisfy a set of skein relations.
Indeed we have:

\begin{proposition}\label{chiskein}
The invariants ${\mathcal X}_S^{\mathtt{B}}(\lambda, \U, \V)$ satisfy the following two skein relations:
\[
 \frac{1}{\sqrt{\lambda_S}} {\mathcal X}_S^{\mathtt{B}}(L_{+}) - \sqrt{\lambda_S} {\mathcal X}_S^{\mathtt{B}}(L_{-}) = \frac{\U -
 \U^{-1}}{d} \sum_{s=0}^{d-1}{\mathcal X}_S^{\mathtt{B}}(L_{s}),
\]
where $L_+ = \widehat{\beta g_i}$, $L_{-} = \widehat{\beta g_i^{-1}}$ and $L_{s} = \widehat{\beta t_i^{s} t_{i+1}^{d-s}}$ with
$\beta = \pi(\alpha)$, $\alpha \in \widehat{W}_n$ and $\pi :\widehat{W}_n \rightarrow \Y$.
\[
{\mathcal X}_S^{\mathtt{B}}(M_{+}) - {\mathcal X}_S^{\mathtt{B}}(M_{-}) = \frac{\V - \V^{-1}}{d} \sum_{s=0}^{d-1}{\mathcal
X}_S^{\mathtt{B}}(M_{s}),
\]
where $M_+ = \widehat{\beta b_i}$, $M_{-} = \widehat{\beta b_i^{-1}}$ and $M_{s} = \widehat{\beta t_i^{s}}$ with $\beta \in
\widehat{W}_n$.
\end{proposition}
\begin{figure}[h]
  \centering
  \includegraphics{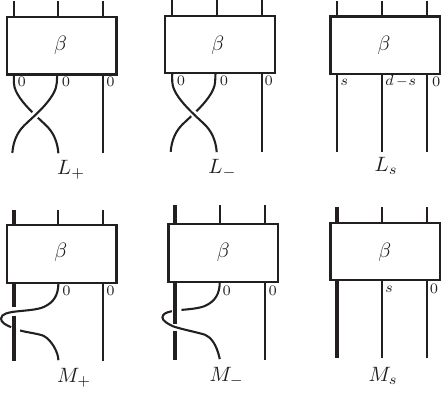}
  \caption{The elements $L_+$, $L{-}$, $L_{s}$, $M_+$, $M_{-}$ and $M_{s}$ in open braid form.}\label{skeinrelfr}
\end{figure}

\begin{proof}
Both skein relations are easily derived from the quadratic relations of $\Y (\U, \V)$. Denote now $\Lambda_S: =\frac{1 -
\lambda_S}{\sqrt{\lambda_S} ( \U - \U^{-1}) {\rm E}_S}$. For the first skein relation we have:
\begin{align*}
{\mathcal X}_S^{\mathtt{B}}(\widehat{\beta g_i^{-1}}) &=  \Lambda_S^{n-1} \left (\sqrt{\lambda_S} \right)^{\varepsilon(\beta-1)}
{\rm Tr}(\beta g_i^{-1})\\
&= \Lambda_S^{n-1} \left (\sqrt{\lambda_S} \right)^{\varepsilon(\beta-1)} {\rm Tr}(\beta g_i) + (\U - \U^{-1}) \Lambda_S^{n-1}
\left (\sqrt{\lambda_S} \right)^{\varepsilon(\beta-1)} {\rm Tr}(\beta e_i)\\
&= \frac{1}{\lambda_S} \Lambda_S^{n-1} \left (\sqrt{\lambda_S} \right)^{\varepsilon(\beta+1)} {\rm Tr}(\beta g_i) + \frac{(\U -
\U^{-1})}{\sqrt{\lambda_S}} \Lambda_S^{n-1} \left (\sqrt{\lambda_S} \right)^{\varepsilon(\beta)} {\rm Tr}(\beta e_i)\\
&= \frac{1}{\lambda_S} {\mathcal X}_S^{\mathtt{B}}(\widehat{\beta g_i}) + \frac{(\U - \U^{-1})}
{d\sqrt{\lambda}}\sum_{s=0}^{d-1}{\mathcal X}_S^{\mathtt{B}} (\beta t_i^s t_{i+1}^{d-s})
\end{align*}
which leads to
\[
 \frac{1}{\sqrt{\lambda_S}} {\mathcal X}_S^{\mathtt{B}}(L_{+}) - \sqrt{\lambda_S} {\mathcal X}_S^{\mathtt{B}}(L_{-}) = \frac{\U -
 \U^{-1}}{d} \sum_{s=0}^{d-1}{\mathcal X}_S^{\mathtt{B}}(L_{s}).
\]
In an analogous way, we prove the second skein relation.
\begin{align*}
{\mathcal X}_S^{\mathtt{B}}(\widehat{\beta b_i^{-1}}) &=  \Lambda_S^{n-1} \left (\sqrt{\lambda_S} \right)^{\varepsilon(\beta)} {\rm
Tr}(\beta b_i^{-1})\\
&= \Lambda_S^{n-1} \left (\sqrt{\lambda_S} \right)^{\varepsilon(\beta)} {\rm Tr}(\beta b_i) + (\V - \V^{-1}) \Lambda_S^{n-1} \left
(\sqrt{\lambda_S} \right)^{\varepsilon(\beta)} {\rm Tr}(\beta f_i)\\
&= \Lambda_S^{n-1} \left (\sqrt{\lambda_S} \right)^{\varepsilon(\beta)} {\rm Tr}(\beta b_i) + (\V - \V^{-1}) \Lambda_S^{n-1} \left
(\sqrt{\lambda_S} \right)^{\varepsilon(\beta)} {\rm Tr}(\beta f_i)\\
&=  {\mathcal X}_S^{\mathtt{B}}(\widehat{\beta b_i}) + \frac{(\V - \V^{-1})} {d}\sum_{s=0}^{d-1}{\mathcal X}_S^{\mathtt{B}} (\beta
t_i^s )
\end{align*}
which is equivalent to:
\[
{\mathcal X}_S^{\mathtt{B}}(M_{+}) - {\mathcal X}_S^{\mathtt{B}}(M_{-}) = \frac{\V - \V^{-1}}{d} \sum_{s=0}^{d-1}{\mathcal
X}_S^{\mathtt{B}}(M_{s}),
\]
\end{proof}
The link invariants on the level of  ${\rm FTL}^{\mathtt{B}}_{d,n}(\U, \V)$ will be specializations
of the invariants ${\mathcal X}_S^{\mathtt{B}}(\lambda, \U, \V)$ for specific values of the trace parameters $x_i$, $y_j$ and $z$. Theorem~\ref{mainthm} provides the conditions so that these new invariants are well-defined. Of course, not all values for $x_i$, $y_j$ and $z$ furnish topologically interesting link invariants and so we shall
use Corollary~\ref{parvals} to filter out such values.

In this context, we discard the cases $1(i)$, $2(i)$ and $2(ii)$ of Corollary~\ref{parvals}. The reason behind this is that if we
specialize the trace parameters in the expression of ${\mathcal X}_S^{\mathtt{B}}$ to any of the cases mentioned just above, we
will obtain an invariant that fails to distinguish basic
pairs of links. In more detail, for $d=1$ we have that $x_k=1$ and so the parameters $z$ and $y_k$ correspond to values  that were
discarded in the
classical case. From the surviving values of Corollary~\ref{parvals}, we deduce that the
rescaling factor $\lambda_S = \U^4$ and so we have:

\begin{definition}\label{phinv}
Let ${\rm X}= ({\rm x}_1 , \ldots , {\rm x}_{d-1})$ a solution of the ${\rm E}$-system, $S \subset C_d$ that
parametrizes said solution. Let also the trace parameters $y_k$ to be as in case 1(ii) of Corollary~\ref{parvals}  and let  $z=
-\frac{1}{\U (\U^2 +1) |S|}$. Then, the following map is an invariant of framed links inside the solid torus:
\[
\rho_S^{\mathtt{B}}(\U, \V) (\widehat{\alpha}):= \left ( - \frac{1 + \U^2}{{\rm E}_S \U} \right )^{n-1} \U^{2\varepsilon(\alpha)}
{\rm Tr}(\bar{\pi} (\alpha)) ={\mathcal X}_S^{\mathtt{B}}(\U^4, \U, \V)  ,
\]
where ${\rm E_S}$, $\varepsilon(\alpha)$ and $\bar{\pi} : \mathcal{F}^{\mathtt{B}}_{d,n} \rightarrow \F (\U ,\V)$ that sends
$\sigma_i \mapsto g_i$ and $t_i \mapsto t_i$.
\end{definition}

Since the invariants $\rho_S^{\mathtt{B}}$ are specializations of ${\mathcal X}_S^{\mathtt{B}}$, they should satisfy also a
specialized version of the skein relations of Proposition~\ref{chiskein}. Indeed, by substituting $\lambda_S=\U^4$ in
Proposition~\ref{chiskein} we obtain:
\begin{proposition}\label{rhoskein}
The invariants $\rho_S^{\mathtt{B}}(\U, \V)$ satisfy the following two skein relations:
\[
\U^{-2} \rho_S^{\mathtt{B}}(L_{+}) - \U^2 \rho_S^{\mathtt{B}}(L_{-}) = \frac{\U - \U^{-1}}{d}
\sum_{s=0}^{d-1}\rho_S^{\mathtt{B}}(L_{s}),
\]
where $L_+ = \widehat{\beta g_i}$, $L_{-} = \widehat{\beta g_i^{-1}}$,$L_{s} = \widehat{\beta t_i^{s} t_{i+1}^{d-s}}$, $\beta =
\pi(\alpha)$, $\alpha \in \widehat{W}_n$ and $\pi :\widehat{W}_n \rightarrow \Y$.
\[
\rho_S^{\mathtt{B}}(M_{+}) - \rho_S^{\mathtt{B}}(M_{-}) = \frac{\V - \V^{-1}}{d} \sum_{s=0}^{d-1}\rho_S^{\mathtt{B}}(M_{s}),
\]
where $M_+ = \widehat{\beta b_i}$, $M_{-} = \widehat{\beta b_i^{-1}}$ and $M_{s} = \widehat{\beta t_i^{s}}$ and $\beta \in
\widehat{W}_n$.
\end{proposition}

\begin{remark}
Notice that for $d=1$, $\rho_S^{\mathtt{B}}(\U, \V)$ coincides with the case of classical links in \eqref{vtypeb}. Moreover, for
$d=1$ the skein relations of Proposition~\ref{chiskein} coincide with the skein relations \eqref{skein1V} and \eqref{skein2V}.
\end{remark}

\subsection{Classical link invariants in the solid torus} \label{sec:geom}
Restricting $\pi$ to classical braids, seen as framed braids with all
framings equal zero, one obtains from $\rho_S^{\mathtt{B}}(\U, \V)$ an invariant for classical links, which is denoted by
$\eta:=\eta_S^{\mathtt{B}}(\U, \V)$. The invariant $\eta$ satisfies the same skein relations as $\rho_S^{\mathtt{B}}(\U, \V)$.
Notice that the algebra ${\rm FTL}_{d,n}(u)$ can be seen as a subalgerba  of $\F$. Indeed, the image of the map
\[
\phi: {\rm FTL}_{d,n}(u) \longrightarrow \F,
\]
 that sends $g_i \mapsto g_i$ and $t_i \mapsto t_i$, is isomorphic to ${\rm FTL}_{d,n}(u)$. Therefore, the trace ${\rm Tr}$, when
 restricted to $\phi({\rm FTL}_{d,n}(u)$), coincides with the trace $tr$ of ${\rm FTL}_{d,n}(u)$.

 A link $L$ inside the solid torus $T$ is called affine if it lies inside a 3-ball $B \subset T$.  Any link in $S^3$ can be seen as
 an embedded affine link in the solid torus. From the above, we can deduce that the invariant $\eta$ contains the invariant
 $\theta_d$ and so it distinguishes at least the same number of non-isotopic links as  $\theta_d$.

 More precisely, the invariant $\theta_d$ distinguishes six pairs of non-isotopic links that are not distinguished by the Jones
 polynomial \cite{gola}. Moreover, $\theta_d$ generalizes to the two-variable link invariant $\theta(q,E)$ that is topologically
 equivalent to the Jones polynomial on knots but stronger than the Jones polynomial on links \cite[Theorem~5]{gola}.  Consequently,
 it is different than the Homflypt and the Kauffman polynomials. It has been shown as well \cite{gola, chlouveraki} that
 $\theta(q,E)$ distinguishes two links from the Eliahou-Kaufmann-Thistlethwaite infinite family of links \cite{ekt} that have the
 same Jones polynomial as the $k$-component unknot.  By specializing $E=1/d$, one can confirm that $\theta_d$ also distinguishes
 these two links. Figure~\ref{fig:pairs} collects all pairs of affine links that are known to be distinguished by the invariant
 $\eta$.

\begin{figure}[h]
  \centering
  \includegraphics[width=\textwidth]{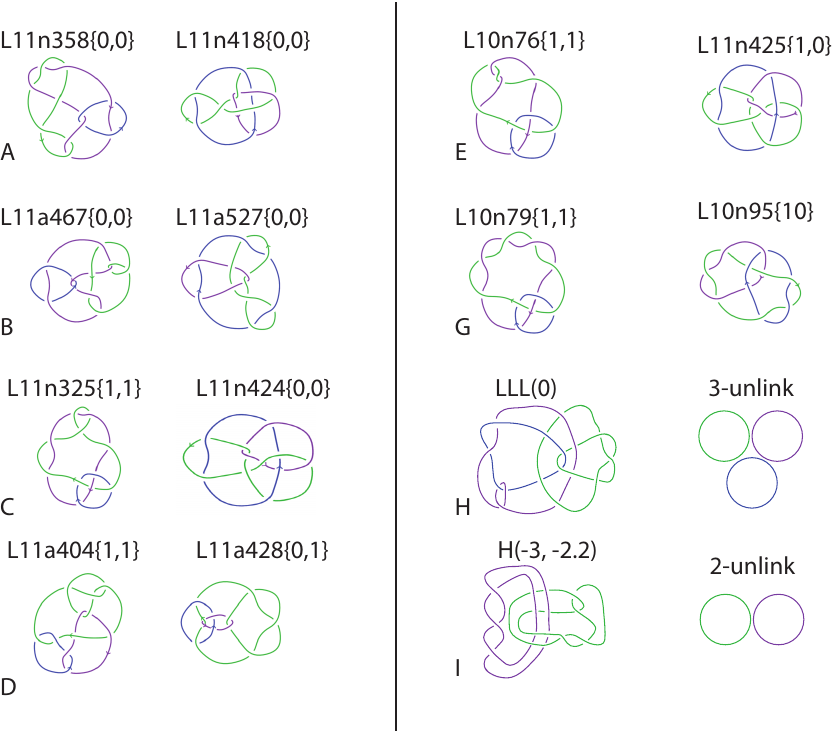}
  \caption{Pairs of affine links  that are distinguished by the invariant $\eta$ and not by the Jones polynomial. Note that pairs H
  and I and unoriented. Pairs A-G follow the Thistelthwaite notation \cite{linkinfo}. Pairs H and I follow the
  Eliahou-Kauffman-Thistelthwaite notation \cite{ekt}.}\label{fig:pairs}
\end{figure}

\subsection{Future work}
The observation that the invariant $\eta$ contains $\theta_d$ suggests that $\eta$ is stronger than the Jones polynomial in the solid
torus, at least on affine links, and that it is different than the Homflypt polynomial in the solid torus. Consider the map $\delta :
\mathbb{K}\widetilde{W}_n \longrightarrow \Y$ that sends $\sigma_i \mapsto g_i$.  In analogy to \cite{chjukala} we have that
$\delta(\mathbb{K} \widetilde{W}_n)$ is isotopic to ${\rm Y}^\mathtt{B}_{(br)}$, the subalgebra of $\Y$ generated only by the
braiding and the looping generators. Note that in $ {\rm Y}^\mathtt{B}_{(br)}$ the generators $t_i$ appear only in the idempotents
$e_i$ and $f_j$ and only after the application of one of the quadratic relations. However, they still have an impact on the skein
relation, as they introduce terms with summations (recall Proposition~\ref{rhoskein}). Unfortunately, this makes difficult to compare
$\eta$ to other invariants in the solid torus on non-affine links.

 In order to overcome this obstacle, we follow the method of \cite{chjukala, gola}.  Let $\mathcal{E}_n^{\mathtt{B}}$ be the algebra
 of braids and ties of type $\mathtt{B}$  \cite{btflores} that is generated by the braiding generators $T_i$ ($i=1 \ldots n-1$) the
 looping generator $B_1$, and the idempotents $E_i$ ($i=1 \ldots n-1$) and $F_j$ ($j=1, \ldots n$).  For $d > n+1$, the map
 $\mathcal{E}_n^{\mathtt{B}} \longrightarrow \Y$ is an embedding \cite{btflores}, which, again in analogy to \cite{chjukala}, implies
 that $\mathcal{E}_n^{\mathtt{B}} \cong {\rm Y}^\mathtt{B}_{(br)}$.

This means that in the context of classical links in the solid torus, seen as closures of framed braids in the solid torus with all
framings equal zero, we can work directly with $\mathcal{E}_n^{\mathtt{B}}$. The advantage is that the framing generators are not
involved in the definition of  $\mathcal{E}_n^{\mathtt{B}}$, which simplifies the corresponding skein relations. For the purpose of
our comparison we aim to generalize the invariant  $\eta$ to a three-variable invariant. One could achieve this by defining the
partition Temperley-Lieb algebra of type $\mathtt{B}$, ${\rm PTL}_n^{\mathtt{B}}$, as an appropriate quotient of
$\mathcal{E}_n^{\mathtt{B}}$ and determine the necessary and sufficient conditions so that the trace ${\rm \mathbf{tr}}_n$ of
$\mathcal{E}_n^{\mathtt{B}}$ passes to ${\rm PTL}_n^{\mathtt{B}}$. Under these conditions, we will obtain the desired generalized
invariant. This is a work in progress and it will be the subject of a sequel paper.

\section*{Acknowledgements}
The authors would like to thank the referee for the careful reading and his/her valuable remarks. This project was partially supported by CONICYT PAI 79140019 and FONDECYT 11170305. The authors would also like to acknowledge the contribution of the COST Action CA17139. 

\bibliography{bibliography}{}
\bibliographystyle{siam}

\end{document}